\documentclass[12pt]{article}

\usepackage[affil-it]{authblk}
      \usepackage{latexsym}
         \usepackage[reqno, namelimits, sumlimits]{amsmath}
         \usepackage{amssymb, amsfonts}
         \usepackage{amsthm}

\newcommand{\X}{X_{p_0,\delta}(T)}

 \newtheorem{theorem}{Theorem}[section]
 
 \newtheorem{definition}[theorem]{Definition}
 
 \newtheorem{lemma}[theorem]{Lemma}
 
 \newtheorem{remark}[theorem]{Remark}

 \newtheorem{cor}[theorem]{Corollary}
 \newtheorem{pro}[theorem]{Proposition}

\title{ Uniqueness Results for Weak Leray-Hopf Solutions of the Navier-Stokes System with Initial Values in Critical Spaces}

\author{T Barker
  \footnote{OxPDE, Mathematical Institute, University of Oxford, Oxford, UK. \\Email address: \texttt{tobiasbarker5@gmail.com}; }}
\date{ \today}

\begin{document}
\maketitle 
\begin{abstract} The main subject  of this paper concerns the establishment of certain classes of initial data, which grant short time uniqueness of the associated weak Leray-Hopf solutions of the three dimensional Navier-Stokes equations. In particular, our main theorem that this holds for any solenodial initial data, with finite $L_2(\mathbb{R}^3)$ norm, that also belongs to
to certain subsets of $VMO^{-1}(\mathbb{R}^3)$. As a corollary of this, we obtain the same conclusion for any solenodial $u_{0}$ belonging to $L_{2}(\mathbb{R}^3)\cap \mathbb{\dot{B}}^{-1+\frac{3}{p}}_{p,\infty}(\mathbb{R}^3)$, for any $3<p<\infty$. Here, $\mathbb{\dot{B}}^{-1+\frac{3}{p}}_{p,\infty}(\mathbb{R}^3)$ denotes the closure of test functions in the critical Besov space ${\dot{B}}^{-1+\frac{3}{p}}_{p,\infty}(\mathbb{R}^3)$.
Our results rely on the establishment  of certain continuity properties near the initial time, for weak Leray-Hopf solutions of the Navier-Stokes equations, with these classes of initial data. Such properties seem to be of independent interest.
Consequently, we are also able to show if a weak Leray-Hopf solution $u$ satisfies certain  extensions of the Prodi-Serrin  condition on $\mathbb{R}^3 \times ]0,T[$, then it is unique on $\mathbb{R}^3 \times ]0,T[$ amongst all other weak Leray-Hopf solutions with the same initial value. In particular, we show this is the case if $u\in L^{q,s}(0,T; L^{p,s}(\mathbb{R}^3))$ or if it's $L^{q,\infty}(0,T; L^{p,\infty}(\mathbb{R}^3))$ norm is sufficiently small, where $3<p< \infty$, $1\leq s<\infty$ and $3/p+2/q=1$.
\end{abstract}
\setcounter{equation}{0}
\setcounter{equation}{0}
\section{Introduction}

This paper concerns  the Cauchy problem for the Navier-Stokes system in the space-time domain $Q_{\infty}:=\mathbb{R}^3\times ]0,\infty[$ for vector-valued function $v=(v_1,v_2,v_3)=(v_i)$ and scalar function $p$, satisfying the equations
\begin{equation}\label{directsystem}
\partial_tv+v\cdot\nabla v-\Delta v=-\nabla p,\qquad\mbox{div}\,v=0
\end{equation}
in $Q_{\infty}$,
with the initial conditions
\begin{equation}\label{directic}
v(\cdot,0)=u_0(\cdot).\end{equation}
This paper will concern a certain class of solutions to 
(\ref{directsystem})-(\ref{directic}), which we will call weak Leray-Hopf solutions. Before defining this class, we introduce some necessary notation. 
 Let $J(\mathbb{R}^3)$ be the closure of $$C^{\infty}_{0,0}(\mathbb{R}^3):=\{u\in C_{0}^{\infty}(\mathbb{R}^3): \rm{div}\,\,u=0\}$$
 with respect to the $L_{2}(\mathbb{R}^3)$ norm.
 Moreover, $\stackrel{\circ}J{^1_2}(\mathbb{R}^3)$ is defined the completion of the space
$C^\infty_{0,0}(\mathbb{R}^3)$
with respect to $L_2$-norm and the Dirichlet integral
$$\Big(\int\limits_{\mathbb{R}^3} |\nabla v|^2dx\Big)^\frac 12 .$$

 Let us now define the notion of 'weak Leray-Hopf solutions' to the Navier-Stokes system.
 \begin{definition}\label{weakLerayHopf}
Consider $0<S\leq \infty$. Let 
 \begin{equation}\label{initialdatacondition}
 u_0\in J(\mathbb{R}^3).
 \end{equation}
 We say that $v$ is a 'weak Leray-Hopf solution' to the Navier-Stokes Cauchy problem in $Q_S:=\mathbb{R}^3 \times ]0,S[$ if it satisfies the following properties:
\begin{equation}
v\in \mathcal{L}(S):= L_{\infty}(0,S; J(\mathbb{R}^3))\cap L_{2}(0,S;\stackrel{\circ} J{^1_2}(\mathbb{R}^3)).
\end{equation} 
Additionally,  for any $w\in L_{2}(\mathbb{R}^3)$:
\begin{equation}\label{vweakcontinuity}
t\rightarrow \int\limits_{\mathbb R^3} w(x)\cdot v(x,t)dx
\end{equation}
is a continuous function on  $[0,S]$ (the semi-open interval should be taken if $S=\infty$).
The Navier-Stokes equations are satisfied by $v$ in a weak sense:
\begin{equation}\label{vsatisfiesNSE}
\int\limits_{0}^{S}\int\limits_{\mathbb{R}^3}(v\cdot \partial_{t} w+v\otimes v:\nabla w-\nabla v:\nabla w) dxdt=0
\end{equation}
for any divergent free test function $$w\in C_{0,0}^{\infty}(Q_{S}):=\{ \varphi\in C_{0}^{\infty}(Q_{S}):\,\,\rm{div}\,\varphi=0\}.$$
The initial condition is satisfied strongly in the $L_{2}(\mathbb{R}^3)$ sense:
\begin{equation}\label{vinitialcondition}
\lim_{t\rightarrow 0^{+}} \|v(\cdot,t)-u_0\|_{L_{2}(\mathbb{R}^3)}=0.
\end{equation}

Finally, $v$ satisfies the energy inequality:
\begin{equation}\label{venergyineq}
\|v(\cdot,t)\|_{L_{2}(\mathbb{R}^3)}^2+2\int\limits_{0}^t\int\limits_{\mathbb{R}^3} |\nabla v(x,t')|^2 dxdt'\leq \|u_0\|_{L_{2}(\mathbb{R}^3)}^2 
\end{equation}
for all $t\in [0,S]$ (the semi-open interval should be taken if $S=\infty$).
\end{definition}
The corresponding global in time existence result, proven in \cite{Le}, is as follows.
\begin{theorem}
Let $u_0\in J(\mathbb{R}^3)$. Then, there exists at least one weak Leray-Hopf solution on $Q_{\infty}$.
\end{theorem}
There are two big open problems concerning weak Leray-Hopf solutions.
\begin{enumerate}
\item
(Regularity)\footnote{This is closely related to one of the Millenium problems, see \cite{fefferman}.} Given any initial data $u_0\in J(\mathbb{R}^3)$, is there a global in time weak Leray-Hopf solution that is regular for all times \footnote{By regular for all time, we mean $C^{\infty}(\mathbb{R}^3 \times ]0,\infty[)$ with every space-time derivative in $L_{\infty}(\epsilon,T; L_{\infty}(\mathbb{R}^3))$ for any $0<\epsilon<T<\infty$.}?
\item
(Uniqueness) Given any initial data $u_0\in J(\mathbb{R}^3)$, is the associated global in time weak Leray-Hopf solution unique in the class of weak Leray-Hopf solutions?
\end{enumerate}
Under certain restrictions of the initial data, it is known since \cite{Le} that 1)(Regularity) implies 2)(Uniqueness)\footnote{The connection made there concerns the slightly narrower class of 'turbulent solutions' defined by Leray in \cite{Le}.}. However, this implication may not be valid for more general classes of initial data. Indeed, certain unverified non uniqueness scenarios, for weak Leray-Hopf solutions, have recently been suggested in \cite{jiasverak2015}. In the scenario suggested there, the non unique solutions are regular. This paper is concerned with the following very natural question arising from 2)(Uniqueness).
\begin{itemize}
\item[]
\textbf{(Q) Which $\mathcal{Z}\subset \mathcal{S}^{'}(\mathbb{R}^3)$ are such that $u_0 \in J(\mathbb{R}^3)\cap\mathcal{Z}$ implies uniqueness of the associated  weak Leray-Hopf solutions on some time interval?}
\end{itemize}
There are a vast number of papers related to this question. We now give some incomplete references, which are directly concerned with this question and closely related to this paper. It was proven in \cite{Le} that for $\mathcal{Z}=\stackrel{\circ} J{^1_2}(\mathbb{R}^3)$ and $\mathcal{Z}= L_{p}(\mathbb{R}^3)$ ($3<p\leq \infty$), we have short time uniqueness in the slightly narrower class of 'turbulent solutions'. The same conclusion was shown to hold in \cite{FJR} for the weak Leray-Hopf class. It was later shown in \cite{Kato} that $\mathcal{Z}= L_{3}(\mathbb{R}^3)$ was sufficient for short time uniqueness of weak Leray-Hopf solutions. At the start of the 21st Century, \cite{GP2000} provided a positive answer for question \textbf{(Q)} for the homogeneous Besov spaces $$\mathcal{Z}= \dot{B}_{p,q}^{-1+\frac{3}{p}}(\mathbb{R}^3)$$ with $p, q<\infty$ and $$\frac{3}{p}+\frac{2}{q}\geq 1.$$
An incomplete selection of further results in this direction are \cite{chemin}, \cite{dongzhang}-\cite{dubois}, \cite{germain} and \cite{LR1}, for example. A more complete history regarding question \textbf{(Q)} can be found in \cite{germain}.

 An approach (which we will refer to as approach 1) to determining $\mathcal{Z}$ such that \textbf{(Q)} is true, was first used for the Navier-Stokes equation in \cite{Le} and is frequently found in the literature. 
 The principle aim of approach 1 is to show for certain $\mathcal{Z}$ and $u_0\in\mathcal{Z}\cap J(\mathbb{R}^3)$, one can construct a weak Leray Hopf solution $V(u_0)$ belonging to a path space $\mathcal{X}_{T}$ having certain features. Specifically, $\mathcal{X}_{T}$ has the property that any weak Leray-Hopf solution (with arbitrary $u_0\in J(\mathbb{R}^3)$ initial data) in $\mathcal{X}_{T}$ is unique amongst all weak Leray-Hopf solutions with the same initial data.


A crucial step in approach 1 is the establishment of appropriate estimates of the trilinear form $F:\mathcal{L}(T)\times\mathcal{L}(T)\times \mathcal{X}_{T}\times ]0,T[\rightarrow \mathbb{R}$ given by:
\begin{equation}\label{trilinearformweakstrong}
F(a,b,c,t):=\int_{0}^{t}\int\limits_{\mathbb{R}^3} (a\otimes c):\nabla  b dyd\tau.
\end{equation} 
As mentioned in \cite{germain}, these estimates of this trilinear form typically play two roles. The first is to provide rigorous justification of the energy inequality for $w:= V(u_0)-u(u_0)$, where $u(u_0)$ is another weak Leray-Hopf solution with the same initial data. The second is to allow the applicability of Gronwall's lemma to infer $w\equiv 0$ on $Q_{T}$. 

The estimates of the trilinear form needed for approach 1 appear to be  restrictive, with regards to  the spaces $\mathcal{Z}$ and $\mathcal{X}_{T}$ that can be considered. Consequently, \textbf{(Q)} has remained open for the Besov spaces $$\mathcal{Z}= \dot{B}_{p,q}^{-1+\frac{3}{p}}(\mathbb{R}^3)$$ with $p\in ]3,\infty[,\,\, q\in [1,\infty[$ and $$\frac{3}{p}+\frac{2}{q}<1.$$
The obstacle of using approach 1 for this case, has been explicitly noted in \cite{GP2000} and \cite{germain}. 
\begin{itemize}
\item[]
'It does not seem possible to improve on the continuity (of the trilinear term) without using in a much deeper way that not only $u$ and $V(u_0)$ are in the Leray class $\mathcal{L}$ but also solutions of the equation.'(\cite{GP2000})
\end{itemize}
For analagous Besov spaces on bounded domains, question \textbf{(Q)} has also been considered recently in \cite{farwiggiga2016}-\cite{farwiggigahsu2016}. There, a restricted version of \textbf{(Q)} is shown to hold. Namely, the authors  prove uniqueness within the subclass of 'well-chosen weak solutions' describing weak Leray-Hopf solutions constructed by concrete approximation procedures. Furthermore, in \cite{farwiggiga2016}-\cite{farwiggigahsu2016} it is explicitly mentioned that a complete answer to \textbf{(Q)} for these cases is 'out of reach'.
  
In this paper, we provide a positive answer to \textbf{(Q)} for $\mathcal{Z}= \mathbb{\dot{B}}^{-1+\frac{3}{p}}_{p,\infty}(\mathbb{R}^3)$, with any $3<p<\infty$ Here, $\mathbb{\dot{B}}^{-1+\frac{3}{p}}_{p,\infty}(\mathbb{R}^3)$ is the closure of smooth compactly supported functions in $\dot{B}^{-1+\frac{3}{p}}_{p,\infty}(\mathbb{R}^3)$ and is such that
$$\dot{B}^{-1+\frac{3}{p}}_{p,p}(\mathbb{R}^3)\hookrightarrow \mathbb{\dot{B}}^{-1+\frac{3}{p}}_{p,\infty}(\mathbb{R}^3).$$
 In fact this is a corollary of our main theorem, which provides a positive answer to \textbf{(Q)} for other classes of $\mathcal{Z}$.
From this point onwards, for $p_0>3$, we will denote $$s_{p_0}:=-1+\frac{3}{p_0}<0.$$
Moreover, for $ 2<\alpha\leq 3$ and $p_{1}>\alpha$, we define $$s_{p_1,\alpha}:= -\frac{3}{\alpha}+\frac{3}{p_1}<0.$$
Now, we state the main theorem of this paper.
\begin{theorem}\label{weakstronguniquenessBesov}
Fix $2<\alpha \leq 3.$ 
\begin{itemize}
\item For $2<\alpha< 3$, take any $p$ such that $\alpha <p< \frac{\alpha}{3-\alpha}$.
\item For $\alpha=3$, take any $p$ such that $3<p<\infty$.
\end{itemize}
Consider a weak Leray-Hopf solution $u$ to the Navier-Stokes system on $Q_{\infty}$, with initial data  $$u_0\in VMO^{-1}(\mathbb{R}^3)\cap \dot{B}^{s_{p,\alpha}}_{p,p}(\mathbb{R}^3)\cap J(\mathbb{R}^3).$$ Then, there exists a $T(u_0)>0$ such that all weak Leray-Hopf solutions on $Q_{\infty}$, with initial data $u_0$, coincide with $u$ on $Q_{T(u_0)}:=\mathbb{R}^3\times ]0,T(u_0)[.$
\end{theorem} 
Let us remark that previous results of this type are given in \cite{chemin} and \cite{dongzhang} respectively, with the additional assumption that $u_0$ belongs to a nonhomogeneous Sobolev space $H^{s}(\mathbb{R}^3)$, with $s>0$.
 By comparison the assumptions of Theorem \ref{weakstronguniquenessBesov} are weaker.
 This follows because of the following embeddings. For $s>0$ there exists $2<\alpha\leq 3$ such that for $p\geq \alpha$:
 $${H}^{s}(\mathbb{R}^3)\hookrightarrow L_{\alpha}(\mathbb{R}^3)\hookrightarrow {\dot{B}^{s_{p,\alpha}}}_{p,p}(\mathbb{R}^3). $$
 \begin{cor}\label{cannoneweakstronguniqueness}
 Let $3<p<\infty$.
 Consider a weak Leray-Hopf solution $u$ to the Navier-Stokes system on $Q_{\infty}$, with initial data  $$u_0\in \dot{\mathbb{B}}_{p,\infty}^{s_{p}}(\mathbb{R}^3)\cap J(\mathbb{R}^3).$$ Then, there exists a $T(u_0)>0$ such that all weak Leray-Hopf solutions on $Q_{\infty}$, with initial data $u_0$, coincide with $u$ on $Q_{T(u_0)}:=\mathbb{R}^3\times ]0,T(u_0)[.$
 \end{cor}

Our main tool to prove Theorem \ref{weakstronguniquenessBesov} is the new observation that weak Leray-Hopf solutions, with this class of initial data, have stronger continuity properties near $t=0$ than general members of the energy class $\mathcal{L}$. In \cite{cheminplanchon},  a similar property was obtained for  for the mild solution with initial data in $\dot{B}^{-\frac{1}{4}}_{4,4}(\mathbb{R}^3)$. Recently, in case of 'global weak $L_3$ solutions' with $L_3(\mathbb{R}^3)$ initial data, properties of this type were established in \cite{sersve2016}. See also \cite{barkerser16} for the case of $L^{3,\infty}$ initial data, in the context of 'global weak $L^{3,\infty}(\mathbb{R}^3)$ solutions'. Let us mention that  throughout this paper, $$S(t)u_0(x):=\Gamma(\cdot,t)\star u_0$$
where $\Gamma(x,t)$ is the kernel for the heat flow in $\mathbb{R}^3$.
Here is our main Lemma.
\begin{lemma}\label{estnearinitialforLeraywithbesov}

Take $\alpha$ and $p$ as in Theorem \ref{weakstronguniquenessBesov}. Assume that
$$
u_0 \in J(\mathbb{R}^3)\cap \dot{B}^{s_{p,\alpha}}_{p,p}(\mathbb{R}^3).
$$ 
Then for any weak Leray-Hopf $u$ solution on $Q_{T}:=\mathbb{R}^3 \times ]0,T[$, with initial data $u_0$, we infer the following.
There exists  $$\beta(p,\alpha)>0$$ and $$\gamma(\|u_{0}\|_{\dot{B}^{s_{p,\alpha}}_{p,p}(\mathbb{R}^3)}, p,\alpha)>0$$ such that  for $t\leq \min({1,\gamma},T):$
\begin{equation}\label{estimatenearinitialtime}
\|u(\cdot,t)-S(t)u_0\|_{L_{2}}^{2}\leq t^{\beta} c(p,\alpha, \|u_0\|_{L_{2}(\mathbb{R}^3)}, \|u_0\|_{\dot{B}^{s_{p,\alpha}}_{p,p}(\mathbb{R}^3)}).
\end{equation}
\end{lemma}
This then allows us to apply a less restrictive version of approach 1. 
Namely, we show that for any initial data in this specific class, there exists a weak Leray-Hopf solution $V(u_0)$ on $Q_T$, which belongs to a path space $\mathcal{X}_{T}$ with the following property. Namely, $\mathcal{X}_{T}$ grants uniqueness for weak Leray-Hopf solutions with the same initial data in this specific class (rather than for arbitrary initial data in $J(\mathbb{R}^3)$, as required in approach 1).
A related strategy has been used in \cite{dongzhang}. However, in \cite{dongzhang} an additional restriction is imposed, requiring that the initial data has  positive Sobolev regularity.

\newpage
\textbf{Remarks}
\begin{enumerate}
\item Another notion of solution, to the Cauchy problem of the Navier Stokes system, was pioneered in \cite{Kato} and \cite{KatoFujita}. These solutions, called 'mild solutions' to the Navier-Stokes system, are constructed using a contraction principle and are unique in their class. Many authors have given classes of initial data for which mild solutions of the Navier-Stokes system exist. See, for example, \cite{cannone}, \cite{GigaMiy1989}, \cite{KozYam1994},  \cite{Plan1996} and \cite{Taylor1992}.

The optimal result in this direction was established in  \cite{kochtataru}. The authors there proved global in time existence of mild solutions for  solenoidal initial data with small $BMO^{-1}(\mathbb{R}^3)$ norm, as well as local in time existence for  solenoidal $u_0\in {VMO}^{-1}(\mathbb{R}^3)$. 
Subsequently, the results of the paper \cite{LRprioux} implied that if $u_0 \in J(\mathbb{R}^3)\cap {VMO}^{-1}(\mathbb{R}^3)$ then the mild solution is a weak Leray-Hopf solution. Consequently, we formulate the following plausible conjecture \textbf{(C)}.
\begin{itemize}
\item[]
\textbf{(C) Question (Q) is affirmative for $\mathcal{Z}={VMO}^{-1}(\mathbb{R}^3)$.}
\end{itemize}
\item 
 In \cite{GregoryNote1}, the following open question was discussed:

\begin{itemize}
\item[] \textbf{(Q.1)
{ Assume that $u_{0k}\in J(\mathbb{R}^3)$ are compactly supported in a fixed compact set and converge to $u_0\equiv 0$ weakly in $L_2(\mathbb{R}^3)$. Let $u^{(k)}$ be the  weak Leray-Hopf solution with the initial value $u_0^{(k)}$. Can we conclude that $v^{(k)}$ converge to $v\equiv 0$ in the sense of distributions? }}
\end{itemize}
In \cite{GregoryNote1} it was shown that (Q.1) holds true
under the following additional restrictions. Namely
\begin{equation}\label{initialdatanormbdd}
\sup_{k}\|u_{0}^{(k)}\|_{L_{s}(\mathbb{R}^3)}<\infty\,\,\,\,\,\,\,\rm{for\,\,some}\,\,3<s\leq\infty
\end{equation}
and that $u^{(k)}$ and it's associated pressure $p^{(k)}$ satisfy the local energy inequality:
\begin{equation}\label{localenergyinequality}
\int\limits_{\mathbb R^3}\varphi(x,t)|u^{(k)}(x,t)|^2dx+2\int\limits_{0}^t\int\limits_{\mathbb R^3}\varphi |\nabla u^{(k)}|^2 dxds\leq$$$$\leq
\int\limits_{0}^{t}\int\limits_{\mathbb R^3}|u^{(k)}|^2(\partial_{t}\varphi+\Delta\varphi)+u^{(k)}\cdot\nabla\varphi(|u^{(k)}|^2+2p^{(k)}) dxds
\end{equation}
for all non negative functions $\varphi\in C_{0}^{\infty}(Q_{\infty}).$ Subsequently,  in \cite{sersve2016} it was shown that the same conclusion holds with (\ref{initialdatanormbdd}) replaced by weaker assumption that
\begin{equation}\label{criticalL3seqbdd}
\sup_{k}\|u_{0}^{(k)}\|_{L_{3}(\mathbb{R}^3)}<\infty.
\end{equation} 
In \cite{barkerser16} , this was further weakened to boundedness of $u_{0}^{(k)}$ in $L^{3,\infty}(\mathbb{R}^3)$.

 Lemma \ref{estnearinitialforLeraywithbesov} has the consequence that  \textbf{(Q.1)} still holds true, if (\ref{initialdatanormbdd}) is replaced by the assumption that $u_{0}^{(k)}$ is bounded in the supercritical Besov spaces $\dot{B}^{s_{p,\alpha}}_{p,p}(\mathbb{R}^3)$\footnote{ with $p$ and $\alpha$ as in Theorem \ref{weakstronguniquenessBesov}}.
Consequently, as the following continuous embedding holds (recall $\alpha\leq p$ and $2\leq \alpha\leq 3$),
$$L_{\alpha}(\mathbb{R}^3)\hookrightarrow {\dot{B}^{s_{p,\alpha}}}_{p,p}(\mathbb{R}^3),$$
we see that this improves the previous assumptions under which \textbf{(Q.1)} holds true. 
\item 
In \cite{Serrin} and \cite{prodi}, it was shown that if $u$ is a weak Leray Hopf solution on $Q_{T}$ and satisfies
\begin{equation}\label{ladyzhenskayaserrinprodi}
u\in L_{q}(0,T; L_{p}(\mathbb{R}^3))\,\,\,\,\,\,\,\,\,\frac{3}{p}+\frac{2}{q}=1\,\, (3<p\leq \infty\,\,\rm{and}\,\, 2\leq q<\infty),
\end{equation}
then  $u$ coincides on $Q_{T}$ with other any weak Leray-Hopf solution with the same initial data. The same conclusion for the endpoint case $u\in L_{\infty}(0,T; L_{3}(\mathbb{R}^3))$ appeared to be much more challenging and was settled in \cite{ESS2003}.  As a consequence of Theorem \ref{weakstronguniquenessBesov}, we are able to extend the uniqueness criterion (\ref{ladyzhenskayaserrinprodi}) for weak Leray-Hopf solutions. Let us state this as a Proposition.
\begin{pro}\label{extendladyzhenskayaserrinprodi}
Suppose $u$ and $v$ are weak Leray-Hopf solutions on $Q_{\infty}$ with the same initial data $u_0 \in J(\mathbb{R}^3)$.
Then there exists a $\epsilon_{*}=\epsilon_{*}(p,q)>0$ such that if either 
\begin{itemize}
\item
\begin{equation}\label{extendladyzhenskayaserrinprodi1}
u \in L^{q,s}(0,T; L^{p,s}(\mathbb{R}^3))\,\,\,\,\,\,\,\,\,\frac{3}{p}+\frac{2}{q}=1
\end{equation}
\begin{equation}\label{integrabilitycondition1}
(3<p< \infty\,,\,\, 2< q<\infty\,\,\rm{and}\,\, 1\leq s<\infty)
\end{equation}
\item or
\begin{equation}\label{extendladyzhenskayaserrinprodi2}
u \in L^{q,\infty}(0,T; L^{p,\infty}(\mathbb{R}^3))\,\,\,\,\,\,\,\,\,\frac{3}{p}+\frac{2}{q}=1
\end{equation}
\begin{equation}\label{integrabilitycondition2}
(3< p< \infty\,,\,\, 2< q <\infty)
\end{equation}
with
\begin{equation}\label{smallness}
\|u\|_{L^{q,\infty}(0,T; L^{p,\infty}(\mathbb{R}^3))}\leq \epsilon_{*},
\end{equation}
then $u\equiv v$ on $Q_{T}:=\mathbb{R}^3 \times ]0,T[$.
\end{itemize}
\end{pro}
 Let us mention that for sufficently small $\epsilon_{*}$, it was shown in \cite{Sohr} that if $u$ is a weak Leray-Hopf solution on $Q_{\infty}$ satisfying either (\ref{extendladyzhenskayaserrinprodi1})-(\ref{integrabilitycondition1}) or (\ref{extendladyzhenskayaserrinprodi2})-(\ref{smallness}), then $u$ is regular\footnote{By regular on $Q_{T}$, we mean $C^{\infty}(\mathbb{R}^3 \times ]0,T[)$ with every space-time derivative in $L_{\infty}(\epsilon,T; L_{\infty}(\mathbb{R}^3))$ for any $0<\epsilon<T$.} on $Q_{T}$. 
 To the best of our knowledge, it was not previously known whether these conditions on $u$ were sufficient to grant uniqueness on $Q_{T}$, amongst all weak Leray Hopf solutions with the same initial value.\par Uniqueness for the endpoint case $(p,q)=(3,\infty)$ of (\ref{extendladyzhenskayaserrinprodi2})-(\ref{smallness}) is simpler and already known. A proof can be found in \cite{LR1}, for example. Hence, we omit this case. 

\end{enumerate}

\setcounter{equation}{0}
\section{Preliminaries}
\subsection{Notation}
In this subsection, we will introduce  notation that will be repeatedly used throughout the rest of the paper. We adopt the usual summation convention throughout the paper .

 
 For arbitrary vectors $a=(a_{i}),\,b=(b_{i})$ in $\mathbb{R}^{n}$ and for arbitrary matrices $F=(F_{ij}),\,G=(G_{ij})$ in $\mathbb{M}^{n}$ we put
 $$a\cdot b=a_{i}b_{i},\,|a|=\sqrt{a\cdot a},$$
 $$a\otimes b=(a_{i}b_{j})\in \mathbb{M}^{n},$$
 $$FG=(F_{ik}G_{kj})\in \mathbb{M}^{n}\!,\,\,F^{T}=(F_{ji})\in \mathbb{M}^{n}\!,$$
 $$F:G=
 F_{ij}G_{ij},\,|F|=\sqrt{F:F},$$
For  spatial domains and space time domains, we will make use of the following notation:
$$B(x_0,R)=\{x\in\mathbb{R}^3: |x-x_0|<R\},$$
$$B(\theta)=B(0,\theta),\,\,\,B=B(1),$$
$$ Q(z_0,R)=B(x_0,R)\times ]t_0-R^2,t_0[,\,\,\, z_{0}=(x_0,t_0),$$
$$Q(\theta)=Q(0,\theta),\,\,Q = Q(1),\,\, Q_{a,b}:=\mathbb{R}^3\times ]a,b[.$$
Here $-\infty\leq a<b\leq \infty.$
In the special cases where $a=0$ we write $Q_{b}:= Q_{0,b}.$

For $\Omega\subseteq\mathbb{R}^3$,  mean values of integrable functions are denoted as follows
$$[p]_{\Omega}=\frac{1}{|\Omega|}\int\limits_{\Omega}p(x)dx.$$
 For, $\Omega\subseteq\mathbb{R}^3$, the space $[C^{\infty}_{0,0}(\Omega)]^{L_{s}(\Omega)}$ is defined to be the closure of $$C^{\infty}_{0,0}(\Omega):=\{u\in C_{0}^{\infty}(\Omega): \rm{div}\,\,u=0\}$$
 with respect to the $L_{s}(\Omega)$ norm.
 For $s=2$, we define
 $$J(\Omega):= [C^{\infty}_{0,0}(\Omega)]^{L_{2}(\Omega)}.$$
 We define $\stackrel{\circ}J{^1_2}(\Omega)$ as the completion of 
$C^\infty_{0,0}(\Omega)$
with respect to $L_2$-norm and the Dirichlet integral
$$\Big(\int\limits_{\Omega} |\nabla v|^2dx\Big)^\frac 12 .$$
If $X$ is a Banach space with norm $\|\cdot\|_{X}$, then $L_{s}(a,b;X)$,  with $a<b$ and $s\in [1,\infty[$,  will denote the usual Banach space of strongly measurable $X$-valued functions $f(t)$ on $]a,b[$ such that 
$$\|f\|_{L_{s}(a,b;X)}:=\left(\int\limits_{a}^{b}\|f(t)\|_{X}^{s}dt\right)^{\frac{1}{s}}<+\infty.$$ 
The usual modification is made if $s=\infty$.
With this notation, we will define 
$$L_{s,l}(Q_{a,b}):= L_{l}(a,b; L_{s}(\mathbb{R}^3)).$$
Let $C([a,b]; X)$ denote the space of continuous $X$ valued functions on $[a,b]$ with usual norm.
In addition, let $C_{w}([a,b]; X)$ denote the space of $X$ valued functions, which are continuous from $[a,b]$ to the weak topology of $X$.\\
We  define the following Sobolev spaces with mixed norms:
$$ W^{1,0}_{m,n}(Q_{a,b})=\{ v\in L_{m,n}(Q_{a,b}): \|v\|_{L_{m,n}(Q_{a,b})}+$$$$+\|\nabla v\|_{L_{m,n}(Q_{a,b})}<\infty\},$$
$$ W^{2,1}_{m,n}(Q_{a,b})=\{ v\in L_{m,n}(Q_{a,b}): \|v\|_{L_{m,n}(Q_{a,b})}+$$$$+\|\nabla^2 v\|_{L_{m,n}(Q_{a,b})}+\|\partial_{t}v\|_{L_{m,n}(Q_{a,b})}<\infty\}.$$

\subsection{Relevant Function Spaces}

\subsubsection{Homogeneous Besov Spaces and $BMO^{-1}$}
 We first introduce the frequency cut off operators of the Littlewood-Paley theory. The definitions we use are contained in \cite{bahourichemindanchin}. For a tempered distribution $f$, let  $\mathcal{F}(f)$ denote its Fourier transform.  Let $C$ be the annulus $$\{\xi\in\mathbb{R}^3: 3/4\leq|\xi|\leq 8/3\}.$$
Let $\chi\in C_{0}^{\infty}(B(4/3))$ and $\varphi\in C_{0}^{\infty}(C)$  be such that
\begin{equation}\label{valuesbetween0and1}
\forall \xi\in\mathbb{R}^3,\,\,0\leq \chi(\xi), \varphi(\xi)\leq 1 ,
\end{equation}
\begin{equation}\label{dyadicpartition}
\forall \xi\in\mathbb{R}^3,\,\,\chi(\xi)+\sum_{j\geq 0} \varphi(2^{-j}\xi)=1
\end{equation}
and
\begin{equation}\label{dyadicpartition.1}
\forall \xi\in\mathbb{R}^3\setminus\{0\},\,\,\sum_{j\in\mathbb{Z}} \varphi(2^{-j}\xi)=1.
\end{equation}
For $a$ being a tempered distribution, let us define for $j\in\mathbb{Z}$:
\begin{equation}\label{dyadicblocks}
\dot{\Delta}_j a:=\mathcal{F}^{-1}(\varphi(2^{-j}\xi)\mathcal{F}(a))\,\,and\,\, \dot{S}_j a:=\mathcal{F}^{-1}(\chi(2^{-j}\xi)\mathcal{F}(a)).
\end{equation}
Now we are in a position to define the homogeneous Besov spaces on $\mathbb{R}^3$. Let $s\in\mathbb{R}$ and $(p,q)\in [1,\infty]\times [1,\infty]$. Then $\dot{B}^{s}_{p,q}(\mathbb{R}^3)$ is the subspace of tempered distributions such that
\begin{equation}\label{Besovdef1}
\lim_{j\rightarrow-\infty} \|\dot{S}_{j} u\|_{L_{\infty}(\mathbb{R}^3)}=0,
\end{equation}
\begin{equation}\label{Besovdef2}
\|u\|_{\dot{B}^{s}_{p,q}(\mathbb{R}^3)}:= \Big(\sum_{j\in\mathbb{Z}}2^{jsq}\|\dot{\Delta}_{j} u\|_{L_p(\mathbb{R}^3)}^{q}\Big)^{\frac{1}{q}}.
\end{equation}
\begin{remark}\label{besovremark1}
 This definition provides a Banach space if $s<\frac{3}{p}$, see \cite{bahourichemindanchin}.
\end{remark}
\begin{remark}\label{besovremark2}
It is known that if $1\leq q_{1}\leq q_{2}\leq\infty$, $1\leq p_1\leq p_2\leq\infty$ and $s\in\mathbb{R}$, then
$$\dot{B}^{s}_{p_1,q_{1}}(\mathbb{R}^3)\hookrightarrow\dot{B}^{s-3(\frac{1}{p_1}-\frac{1}{p_2})}_{p_2,q_{2}}(\mathbb{R}^3).$$
\end{remark}
\begin{remark}\label{besovremark3}
It is known that for $s=-2s_{1}<0$ and $p,q\in [1,\infty]$, the norm can be characterised by the heat flow. Namely there exists a $C>1$ such that for all $u\in\dot{B}^{-2s_{1}}_{p,q}(\mathbb{R}^3)$:
$$C^{-1}\|u\|_{\dot{B}^{-2s_{1}}_{p,q}(\mathbb{R}^3)}\leq \|\|t^{s_1} S(t)u\|_{L_{p}(\mathbb{R}^3)}\|_{L_{q}(\frac{dt}{t})}\leq C\|u\|_{\dot{B}^{-2s_{1}}_{p,q}(\mathbb{R}^3)}.$$  Here, $$S(t)u(x):=\Gamma(\cdot,t)\star u$$
where $\Gamma(x,t)$ is the kernel for the heat flow in $\mathbb{R}^3$.
\end{remark}
We will also need the following Proposition, whose statement and proof can be found in \cite{bahourichemindanchin} (Proposition 2.22 there) for example. In the Proposition below we use the notation
\begin{equation}\label{Sh}
\mathcal{S}_{h}^{'}:=\{ \rm{ tempered\,\,distributions}\,\, u\rm{\,\,\,such\,\,that\,\,} \lim_{j\rightarrow -\infty}\|S_{j}u\|_{L_{\infty}(\mathbb{R}^3)}=0\}.
\end{equation}
\begin{pro}\label{interpolativeinequalitybahourichemindanchin}
A constant $C$ exists with the following properties. If $s_{1}$ and $s_{2}$ are real numbers such that $s_{1}<s_{2}$ and $\theta\in ]0,1[$, then we have, for any $p\in [1,\infty]$ and any $u\in \mathcal{S}_{h}^{'}$,
\begin{equation}\label{interpolationactual}
\|u\|_{\dot{B}_{p,1}^{\theta s_{1}+(1-\theta)s_{2}}(\mathbb{R}^3)}\leq \frac{C}{s_2-s_1}\Big(\frac{1}{\theta}+\frac{1}{1-\theta}\Big)\|u\|_{\dot{B}_{p,\infty}^{s_1}(\mathbb{R}^3)}^{\theta}\|u\|_{\dot{B}_{p,\infty}^{s_2}(\mathbb{R}^3)}^{1-\theta}.
\end{equation}
\end{pro}

Finally, $BMO^{-1}(\mathbb{R}^3)$ is the space of all tempered distributions such that the following norm is finite:
\begin{equation}\label{bmo-1norm}
\|u\|_{BMO^{-1}(\mathbb{R}^3)}:=\sup_{x\in\mathbb{R}^3,R>0}\frac{1}{|B(0,R)|}\int\limits_0^{R^2}\int\limits_{B(x,R)} |S(t)u|^2 dydt.
\end{equation}
Note that $VMO^{-1}(\mathbb{R}^3)$ is the subspace that coincides with the closure of test functions $C_{0}^{\infty}(\mathbb{R}^3)$, with respect to the norm (\ref{bmo-1norm}).

\subsubsection{Lorentz spaces}
Given a measurable subset $\Omega\subset\mathbb{R}^{n}$, let us define the Lorentz spaces. 
For a measurable function $f:\Omega\rightarrow\mathbb{R}$ define:
\begin{equation}\label{defdist}
d_{f,\Omega}(\alpha):=\mu(\{x\in \Omega : |f(x)|>\alpha\}),
\end{equation}
where $\mu$ denotes Lebesgue measure.
 The Lorentz space $L^{p,q}(\Omega)$, with $p\in [1,\infty[$, $q\in [1,\infty]$, is the set of all measurable functions $g$ on $\Omega$ such that the quasinorm $\|g\|_{L^{p,q}(\Omega)}$ is finite. Here:

\begin{equation}\label{Lorentznorm}
\|g\|_{L^{p,q}(\Omega)}:= \Big(p\int\limits_{0}^{\infty}\alpha^{q}d_{g,\Omega}(\alpha)^{\frac{q}{p}}\frac{d\alpha}{\alpha}\Big)^{\frac{1}{q}},
\end{equation}
\begin{equation}\label{Lorentznorminfty}
\|g\|_{L^{p,\infty}(\Omega)}:= \sup_{\alpha>0}\alpha d_{g,\Omega}(\alpha)^{\frac{1}{p}}.
\end{equation}\\
It is known there exists a norm, which is equivalent to the quasinorms defined above, for which $L^{p,q}(\Omega)$ is a Banach space. 
For $p\in [1,\infty[$ and $1\leq q_{1}< q_{2}\leq \infty$, we have the following continuous embeddings 
\begin{equation}\label{Lorentzcontinuousembedding}
L^{p,q_1}(\Omega) \hookrightarrow  L^{p,q_2}(\Omega)
\end{equation}
and the inclusion is known to be strict.

Let $X$ be a Banach space with norm $\|\cdot\|_{X}$, $ a<b$, $p\in [1,\infty[$ and  $q\in [1,\infty]$. 
 Then $L^{p,q}(a,b;X)$   will denote the space of strongly measurable $X$-valued functions $f(t)$ on $]a,b[$ such that 

\begin{equation}\label{Lorentznormbochner}
\|f\|_{L^{p,q}(a,b; X)}:= \|\|f(t)\|_{X}\|_{L^{p,q}(a,b)}<\infty.
\end{equation}
In particular, if $1\leq q_{1}< q_{2}\leq \infty$, we have the following continuous embeddings 
\begin{equation}\label{Bochnerlorentzcontinuousembedding}
L^{p,q_1}(a,b; X) \hookrightarrow  L^{p,q_2}(a,b; X)
\end{equation}
and the inclusion is known to be strict.

Let us recall a known Proposition known as 'O'Neil's convolution inequality' (Theorem 2.6 of \cite{O'Neil}), which will be used in proving Proposition \ref{extendladyzhenskayaserrinprodi}.
\begin{pro}\label{O'Neil}
Suppose $1\leq p_{1}, p_{2}, q_{1}, q_{2}, r\leq\infty$ are such that
\begin{equation}\label{O'Neilindices1}
\frac{1}{r}+1=\frac{1}{p_1}+\frac{1}{p_{2}}
\end{equation}
and
\begin{equation}\label{O'Neilindices2}
\frac{1}{q_1}+\frac{1}{q_{2}}\geq \frac{1}{s}.
\end{equation}
Suppose that
\begin{equation}\label{fghypothesis}
f\in L^{p_1,q_1}(\mathbb{R}^{n})\,\,\rm{and}\,\,g\in  L^{p_2,q_2}(\mathbb{R}^{n}).
\end{equation}
Then it holds that
\begin{equation}\label{fstargconclusion1}
f\star g \in L^{r,s}(\mathbb{R}^n)\,\,\rm{with} 
\end{equation}
\begin{equation}\label{fstargconclusion2}
\|f\star g \|_{L^{r,s}(\mathbb{R}^n)}\leq 3r \|f\|_{L^{p_1,q_1}(\mathbb{R}^n)} \|g\|_{L^{p_2,q_2}(\mathbb{R}^n)}. 
\end{equation}
\end{pro}
Let us finally state and proof a simple Lemma, which we will make use of in proving Proposition \ref{ladyzhenskayaserrinprodi}.
\begin{lemma}\label{pointwiselorentzdecreasing}
Let  $f:]0,T[\rightarrow ]0,\infty[$ be a function satisfying the following property. In particular, suppose that there exists a $C\geq 1$ such that for any $0<t_{0}\leq t_{1}<T$ we have 
\begin{equation}\label{almostdecreasing}
f(t_{1})\leq C f(t_{0}).
\end{equation}
In addition, assume that for some $1\leq r<\infty$:
\begin{equation}\label{florentz}
f \in L^{r,\infty}(0,T).
\end{equation}
Then one can conclude that for all $t\in ]0,T[$: 
\begin{equation}\label{fpointwise}
f(t)\leq \frac{2C\|f\|_{L^{r,\infty}(0,T)}}{t^{\frac{1}{r}}}.
\end{equation}

\end{lemma}
 \begin{proof}
 It suffices to proof that if $f$ satisfies the hypothesis of Lemma \ref{pointwiselorentzdecreasing},  along with the additional constraint
 \begin{equation}\label{lorentznormhalf}
 \|f\|_{L^{r,\infty}(0,T)}=\frac{1}{2},
 \end{equation}
 then we must necessarily have that for any $0<t<T$
 \begin{equation}\label{fpointwisereduction}
 f(t)\leq \frac{C}{t^{\frac{1}{r}}}.
 \end{equation}
  The assumption (\ref{lorentznormhalf}) implies that
 \begin{equation}\label{weakcharacterisationdistribution}
 \sup_{\alpha>0} \alpha^r \mu(\{s\in]0,T[\,\,\rm{such\,\,that}\,\,f(s)>\alpha\})<1.
 \end{equation}
 Fixing  $t\in ]0,T[$ and setting $\alpha=\frac{1}{t^{\frac{1}{r}}}$, we see that
 \begin{equation}\label{weakcharacterisationdistributiont}
   \mu(\{s\in]0,T[\,\,\rm{such\,\,that}\,\,f(s)>1/{t^{\frac{1}{r}}}\})<t.
 \end{equation}
 For $0<t_{0}\leq t$ we have
 \begin{equation}\label{almostdecreasingrecall}
 f(t)\leq Cf(t_0).
 \end{equation}
 This, together with (\ref{weakcharacterisationdistributiont}), implies (\ref{fpointwisereduction}).
 \end{proof}
\subsection{Decomposition of Homogeneous Besov Spaces}
Next state and prove certain decompositions  for homogeneous Besov spaces. This will play a crucial role in the proof of Theorem \ref{weakstronguniquenessBesov}.  In the context of Lebesgue spaces,  an analogous statement  is Lemma II.I proven by Calderon in \cite{Calderon90}. 
Before stating and proving this, we take note of a useful Lemma presented in \cite{bahourichemindanchin} (specifically, Lemma 2.23 and Remark 2.24 in \cite{bahourichemindanchin}).
\begin{lemma}\label{bahouricheminbook}
Let $C^{'}$ be an annulus and let $(u^{(j)})_{j\in\mathbb{Z}}$ be a sequence of functions such that
\begin{equation}\label{conditionsupport}
\rm{Supp}\, \mathcal{F}(u^{(j)})\subset 2^{j}C^{'} 
\end{equation}
and
\begin{equation}\label{conditionseriesbound}
\Big(\sum_{j\in\mathbb{Z}}2^{jsr}\|u^{(j)}\|_{L_p}^{r}\Big)^{\frac{1}{r}}<\infty.
\end{equation}
Moreover, assume in addition that 
\begin{equation}\label{indicescondition}
s<\frac{3}{p}.
\end{equation}
Then the following holds true.
The series $$\sum_{j\in\mathbb{Z}} u^{(j)}$$ converges (in the sense of tempered distributions) to some $u\in \dot{B}^{s}_{p,r}(\mathbb{R}^3)$, which satisfies the following estimate:
\begin{equation}\label{besovboundlimitfunction}
\|u\|_{\dot{B}^{s}_{p,r}}\leq C_s \Big(\sum_{j\in\mathbb{Z}}2^{jsr}\|u^{(j)}\|_{L_p}^{r}\Big)^{\frac{1}{r}}.
\end{equation}

\end{lemma}
Now, we can state the proposition regarding decomposition of homogeneous Besov spaces. Note that decompositions of a similar type can be obtained abstractly from  real interpolation theory, applied to homogeneous Besov spaces. See Chapter 6 of \cite{berghlofstrom}, for example.
\begin{pro}\label{Decompgeneral}
For $i=1,2,3$ let $p_{i}\in ]1,\infty[$, $s_i\in \mathbb{R}$ and $\theta\in ]0,1[$ be such that $s_1<s_0<s_2$ and $p_2<p_0<p_1$. In addition, assume the following relations hold:
\begin{equation}\label{sinterpolationrelation}
s_1(1-\theta)+\theta s_2=s_0,
\end{equation}
\begin{equation}\label{pinterpolationrelation}
\frac{1-\theta}{p_1}+\frac{\theta}{p_2}=\frac{1}{p_0},
\end{equation}
\begin{equation}\label{besovbanachcondition}
{s_i}<\frac{3}{p_i}.
\end{equation}
Suppose that $u_0\in \dot{B}^{{s_{0}}}_{p_0,p_0}(\mathbb{R}^3).$
Then for all $\epsilon>0$ there exists $u^{1,\epsilon}\in \dot{B}^{s_{1}}_{p_1,p_1}(\mathbb{R}^3)$, $u^{2,\epsilon}\in \dot{B}^{s_{2}}_{p_2,p_2}(\mathbb{R}^3)$ such that 
\begin{equation}\label{udecompgeneral}
u= u^{1,\epsilon}+u^{2,\epsilon},
\end{equation}
\begin{equation}\label{u_1estgeneral}
\|u^{1,\epsilon}\|_{\dot{B}^{s_{1}}_{p_1,p_1}}^{p_1}\leq \epsilon^{p_1-p_0} \|u_0\|_{\dot{B}^{s_{0}}_{p_0,p_0}}^{p_0},
\end{equation}
\begin{equation}\label{u_2estgeneral}
\|u^{2,\epsilon}\|_{\dot{B}^{s_{2}}_{p_2,p_2}}^{p_2}\leq C( p_0,p_1,p_2, \|\mathcal{F}^{-1}\varphi\|_{L_1})\epsilon^{p_2-p_0} \|u_0\|_{\dot{B}^{s_{0}}_{p_0,p_0}}^{p_0}
.\end{equation}
\end{pro}
\begin{proof}
Denote, $$f^{(j)}:= \dot{\Delta}_{j} u,$$ $$f^{(j)N}_{-}:=f^{(j)}\chi_{|f^{(j)}|\leq N} $$ and $$f^{(j)N}_{+}:=f^{(j)}(1-\chi_{|f^{(j)}|\leq N}). $$ 
It is easily verified that the following holds:
$$ \|f^{(j)N}_{-}\|_{L_{p_1}}^{p_1}\leq N^{p_1-p_0}\|f^{(j)}\|_{L_{p_0}}^{p_0},$$
$$\|f^{(j)N}_{+}\|_{L_{p_2}}^{p_2}\leq N^{p_2-p_0}\|f^{(j)}\|_{L_{p_0}}^{p_0}.$$
Thus, we may write
\begin{equation}\label{truncationest1}
2^{p_1 s_1 j}\|f^{(j)N}_{-}\|_{L_{p_1}}^{p_1}\leq N^{p_1-p_0}2^{(p_1 s_1-p_0 s_0)j} 2^{p_0 s_0 j}\|f^{(j)}\|_{L_{p_0}}^{p_0}
\end{equation}
\begin{equation}\label{truncationest2}
2^{p_2 s_2 j}\|f^{(j)N}_{+}\|_{L_{p_2}}^{p_2}\leq N^{p_2-p_0} 2^{(p_2 s_2-p_0s_0)j} 2^{p_0s_0j}\|f^{(j)}\|_{L_{p_0}}^{p_0}.
\end{equation}
With (\ref{truncationest1}) in mind, we define 
$$N({j,\epsilon,s_0,s_1,p_0,p_1}):= \epsilon 2^{\frac{(p_0s_0-p_1s_1)j}{p_1-p_0}}.$$
For the sake of brevity we will write $N({j},\epsilon)$.
Using the relations of the Besov indices given by (\ref{sinterpolationrelation})-(\ref{pinterpolationrelation}), we can infer that
$$N({j,\epsilon})^{p_2-p_0} 2^{(p_2 s_2-p_0s_0)j} = \epsilon^{p_2-p_0}.$$
 The crucial point being that this is independent of $j$.
 Thus, we infer
 \begin{equation}\label{truncationest1.1}
 2^{p_1 s_1 j}\|f^{(j)N(j,\epsilon)}_{-}\|_{L_{p_1}}^{p_1}\leq \epsilon^{p_1-p_0} 2^{p_0 s_0 j}\|f^{(j)}\|_{L_{p_0}}^{p_0},
 \end{equation}
\begin{equation}\label{truncationest2.1}
 2^{p_2 s_2 j}\|f^{(j)N(j,\epsilon)}_{+}\|_{L_{p_2}}^{p_2}\leq \epsilon^{p_2-p_0} 2^{p_0s_0j}\|f^{(j)}\|_{L_{p_0}}^{p_0}.
 \end{equation}
Next, it is well known that for any $ u\in \dot{B}^{s_0}_{p_0,p_0}(\mathbb{R}^3)$ we have that
$\sum_{j=-m}^{m} \dot{\Delta}_{j} u$ converges to $u$ in the sense of tempered distributions.
Furthermore, we have  that $\dot{\Delta}_{j}\dot{\Delta}_{j'} u=0$ if $|j-j'|>1.$ Combing these two facts allows us to observe that
\begin{equation}\label{smoothingtruncations}
\dot{\Delta}_{j} u= \sum_{|m-j|\leq 1} \dot{\Delta}_{m} f^{(j)}= \sum_{|m-j|\leq 1}\dot{\Delta}_{m} f_{-}^{(j)N(j,\epsilon)}+\sum_{|m-j|\leq 1}\dot{\Delta}_{m} f_{+}^{(j)N(j,\epsilon)}.
\end{equation}
Define 
\begin{equation}\label{decomp1eachpiece}
u^{1,\epsilon}_{j}:= \sum_{|m-j|\leq 1}\dot{\Delta}_{m} f_{-}^{(j)N(j,\epsilon)},
\end{equation}
\begin{equation}\label{decomp2eachpiece}
u^{2,\epsilon}_{j}:= \sum_{|m-j|\leq 1}\dot{\Delta}_{m} f_{+}^{(j)N(j,\epsilon)}
\end{equation}
It is clear, that \begin{equation}\label{fouriersupport}
\rm{Supp}\,\mathcal{F}(u^{1,\epsilon}_{j}), \rm{Supp}\,\mathcal{F}(u^{2,\epsilon}_{j})\subset 2^{j}C^{'}.
\end{equation}
Here, $C'$ is the annulus defined by $C':=\{\xi\in\mathbb{R}^3: 3/8\leq |\xi|\leq 16/3\}.$
Using, (\ref{truncationest1.1})-(\ref{truncationest2.1}) we can obtain
the following estimates:
\begin{equation}\label{decomp1est}
2^{p_1 s_1 j}\|u^{1,\epsilon}_{j}\|_{L_{p_1}}^{p_1}\leq \lambda_{1}(p_1, \|\mathcal{F}^{-1} \varphi\|_{L_1}) 2^{p_1 s_1 j}\|f^{(j)N(j,\epsilon)}_{-}\|_{L_{p_1}}^{p_1}\leq$$$$\leq \lambda_{1}(p_1, \|\mathcal{F}^{-1} \varphi\|_{L_1}) \epsilon^{p_1-p_0} 2^{p_0 s_0 j}\|f^{(j)}\|_{L_{p_0}}^{p_0},
\end{equation}
\begin{equation}\label{decomp2est}
2^{p_2 s_2 j}\|u^{2,\epsilon}_{j}\|_{L_{p_2}}^{p_2}\leq \lambda_{2}(p_2, \|\mathcal{F}^{-1} \varphi\|_{L_1}) 2^{p_2 s_2 j}\|f^{(j)N(j,\epsilon)}_{-}\|_{L_{p_2}}^{p_2}\leq$$$$\leq \lambda_{2}(p_2, \|\mathcal{F}^{-1} \varphi\|_{L_1})\epsilon^{p_2-p_0} 2^{p_0 s_0 j}\|f^{(j)}\|_{L_{p_0}}^{p_0}.
\end{equation}
It is then the case that (\ref{fouriersupport})-(\ref{decomp2est}) allow us to apply the results of Lemma \ref{bahouricheminbook}. This allows us to achieve the desired decomposition with the choice
$$u^{1,\epsilon}=\sum_{j\in\mathbb{Z}}u^{1,\epsilon}_{j},$$
$$u^{2,\epsilon}=\sum_{j\in\mathbb{Z}}u^{2,\epsilon}_{j}.$$

\end{proof}
\begin{cor}\label{Decomp}
Fix $2<\alpha \leq 3.$
\begin{itemize}
\item For $2<\alpha< 3$, take $p$ such that $\alpha <p< \frac{\alpha}{3-\alpha}$.
\item For $\alpha=3$, take $p$ such that $3<p<\infty$.
\end{itemize}
For $p$ and $\alpha$ satisfying these conditions, suppose that 
\begin{equation}
u_0\in \dot{B}^{s_{p,\alpha}}_{p,p}(\mathbb{R}^3)\cap L_{2}(\mathbb{R}^3)
\end{equation}
and
 $\rm{div}\,\,u_0=0$ in weak sense. \\
  Then the above assumptions imply that there exists $\max{(p,4)}<p_0<\infty$ and $\delta>0$ such that  for any $\epsilon>0$ there exists weakly divergence free functions 
$\bar{u}^{1,\epsilon}\in \dot{B}^{s_{p_0}+\delta}_{p_0,p_0}(\mathbb{R}^3)\cap L_{2}(\mathbb{R}^3)$ and $\bar{u}^{2,\epsilon}\in L_2(\mathbb{R}^3)$ such that
 
\begin{equation}\label{udecomp1}
u_0= \bar{u}^{1,\epsilon}+\bar{u}^{2,\epsilon},
\end{equation}
\begin{equation}\label{baru_1est}
\|\bar{u}^{1,\epsilon}\|_{\dot{B}^{s_{p_0}+\delta}_{p_0,p_0}}^{p_0}\leq \epsilon^{p_0-p} \|u_0\|_{\dot{B}^{s_{p,\alpha}}_{p,p}}^{p},
\end{equation}
\begin{equation}\label{baru_2est}
\|\bar{u}^{2,\epsilon}\|_{L_2}^2\leq C(p,p_0,\|\mathcal{F}^{-1}\varphi\|_{L_1}) \epsilon^{2-p}\|u_0\|_{\dot{B}^{s_{p,\alpha}}_{p,p}}^p
,
\end{equation} 
\begin{equation}\label{baru_1est.1}
\|\bar{u}^{1,\epsilon}\|_{L_2}\leq  C(\|\mathcal{F}^{-1}\varphi\|_{L_1})\|u_0\|_{L_{2}}.
\end{equation}

\end{cor}
\begin{proof}
\begin{itemize}
\item[]
\textbf{ First case: $2<\alpha< 3$ and $\alpha <p< \frac{\alpha}{3-\alpha}$}
\end{itemize}
Under this condition, we can find $\max{(4,p)}<p_{0}<\infty$ such that
\begin{equation}\label{condition}
\theta:= \frac{\frac{1}{p}-\frac{1}{p_0}}{\frac{1}{2}-\frac{1}{p_0}}>\frac{6}{\alpha}-2.
\end{equation}
Clearly, $0<\theta<1$ and moreover
\begin{equation}\label{summabilityindicerelation}
\frac{1-\theta}{p_0}+\frac{\theta}{2}=\frac{1}{p}.
\end{equation}
Define
\begin{equation}\label{deltadef}
\delta:=\frac{1-\frac{3}{\alpha}+\frac{\theta}{2}}{1-\theta}.
\end{equation}
From (\ref{condition}), we see that $\delta>0$.
One can also see we have the following relation: 
\begin{equation}\label{regularityindicerelation}
(1-\theta)(s_{p_0}+\delta)=s_{p,\alpha}.
\end{equation}

The above relations allow us  to apply Proposition \ref{Decompgeneral} to obtain the following decomposition:
(we note that $\dot{B}^{0}_{2,2}(\mathbb{R}^3)$ coincides with $L_2(\mathbb{R}^3)$ with equivalent norms)
\begin{equation}\label{udecomp}
u_0= {u}^{1,\epsilon}+{u}^{2,\epsilon},
\end{equation}
\begin{equation}\label{u_1est}
\|{u}^{1,\epsilon}\|_{\dot{B}^{s_{p_0}+\delta}_{p_0,p_0}}^{p_0}\leq \epsilon^{p_0-p} \|u_0\|_{\dot{B}^{s_{p,\alpha}}_{p,p}}^p,
\end{equation}
\begin{equation}\label{u_2est}
\|{u}^{2,\epsilon}\|_{L_2}^2\leq C(p,p_0,\|\mathcal{F}^{-1}\varphi\|_{L_1})\epsilon^{2-p} \|u_0\|_{\dot{B}^{s_{p,\alpha}}_{p,p}}^p
.
\end{equation} 
For $j\in\mathbb{Z}$ and $m\in\mathbb{Z}$, it can be seen that  \begin{equation}\label{besovpersistency}
\|\dot{\Delta}_{m}\left( (\dot{\Delta}_{j}u_{0})\chi_{|\dot{\Delta}_{j}u_{0}|\leq N(j,\epsilon)}\right)\|_{L_{2}}\leq C(\|\mathcal{F}^{-1} \varphi\|_{L_1})) \|\dot{\Delta}_{j} u_0\|_{L_2}.
\end{equation}
It is known that  $u_0\in L_{2}$ implies
$$ \|u_0\|_{L_2}^2=\sum_{j\in\mathbb{Z}}\|\dot{\Delta}_{j} u_0\|_{L_2}^2.$$
Using this, (\ref{besovpersistency}) and the definition of $u^{1,\epsilon}$ from Proposition \ref{Decompgeneral}, we can infer that
$$\|{u}^{1,\epsilon}\|_{L_2}\leq   C(\|\mathcal{F}^{-1} \varphi\|_{L_1})\|u_0\|_{L_{2}}.$$

To establish the decomposition of the Corollary we
 apply the Leray projector to each of $u^{1,\epsilon}$ and $u^{2,\epsilon}$, which is a continuous linear operator on the homogeneous  Besov spaces under consideration.
 \begin{itemize}
 \item[] \textbf{ Second case: $\alpha=3$ and $3<p<\infty$}
 \end{itemize}
 In the second case, we choose any $p_0$ such that $\max{(4,p)}<p_0<\infty$. With this $p_0$ we choose $\theta$ such that
 $$
\frac{1-\theta}{p_0}+\frac{\theta}{2}=\frac{1}{p} .$$
If we define 
\begin{equation}
\delta:=\frac{\theta}{2(1-\theta)}>0,
\end{equation}
we see that
\begin{equation}\label{regulairtyindexrelation}
(s_{p_0}+\delta)(1-\theta)= s_p.
\end{equation}
These relations allow us to obtain the decomposition of the Corollary, by means of identical arguments to those presented in the the first case of this proof.

\end{proof}
\setcounter{equation}{0}
\section{Some Estimates Near the Initial Time for Weak Leray-Hopf Solutions}
\subsection{Construction of  Mild Solutions with Subcritical Besov Initial Data}
Let $\delta>0$ be such that $s_{p_0}+\delta<0$ and define the space
$$X_{p_0,\delta}(T):=\{f\in \mathcal{S}^{'}(\mathbb{R}^3\times ]0,T[): \sup_{0<t<T} t^{-\frac{s_{p_0}}{2}-\frac{\delta}{2}}\|f(\cdot,t)\|_{L_{p_0}(\mathbb{R}^3)}<\infty\}.$$
From remarks \ref{besovremark2} and \ref{besovremark3}, we observe that \begin{equation}\label{Besovembedding}
u_0 \in \dot{B}^{s_{p_0}+\delta}_{p_0,p_0}(\mathbb{R}^3)\Rightarrow \|S(t)u_{0}\|_{X_{p_0,\delta}(T)}\leq C\|u_0\|_{\dot{B}^{s_{p_0}+\delta}_{p_0,\infty}}\leq C\|u_0\|_{\dot{B}^{s_{p_0}+\delta}_{p_0,p_0}}.
\end{equation}
In this subsection, we construct \footnote{ This is in a similar spirit to a construction contained in a joint work with Gabriel Koch, to appear.}  mild solutions with weakly divergence free initial data in $\dot{B}^{s_{p_0}+\delta}_{p_0,p_0}(\mathbb{R}^3)\cap J(\mathbb{R}^3)$.   
Before constructing mild solutions we will briefly explain the relevant kernels and their pointwise estimates.

 Let us consider the following Stokes problem:
$$\partial_t v-\Delta v +\nabla q=-\rm{div}\,F,\qquad {\rm div }\,v=0$$
in $Q_T$,
$$v(\cdot,0)=0.$$

Furthermore, assume that $F_{ij}\in C^{\infty}_{0}(Q_T).$
 Then a formal solution to the above initial boundary value problem has the form:
$$v(x,t)=\int \limits_0^t\int \limits_{\mathbb R^3}K(x-y,t-s):F(y,s)dyds.$$
The kernel $K$ is derived with the help of the heat kernel $\Gamma$ as follows:
$$\Delta_{y}\Phi(y,t)=\Gamma(y,t),$$
$$K_{mjs}(y,t):=\delta_{mj}\frac{\partial^3\Phi}{\partial y_i\partial y_i\partial y_s}(y,t)-\frac{\partial^3\Phi}{\partial y_m\partial y_j\partial y_s}(y,t).$$
Moreover, the following pointwise estimate is known:
\begin{equation}\label{kenrelKest}
|K(x,t)|\leq\frac{C}{(|x|^2+t)^2}.
\end{equation}
Define
\begin{equation}\label{fluctuationdef}
 G(f\otimes g)(x,t):=\int\limits_0^t\int\limits_{\mathbb{R}^3} K(x-y,t-s): f\otimes g(y,s) dyds.
 \end{equation}

\begin{theorem}\label{regularity} 
 Consider $p_0$ and $\delta$ such that $4<p_0<\infty$, $\delta>0$ and $s_{p_0}+\delta<0$. Suppose that $u_0\in \dot{B}^{s_{p_0}+\delta}_{p_0,p_0}(\mathbb{R}^3)\cap J(\mathbb{R}^3)$.
There exists a constant $c=c(p_0)$ such that 
if 
\begin{equation}\label{smallnessasmp}
4cT^{\frac{\delta}{2}}\|u_{0}\|_{\dot{B}^{s_{p_0}+\delta}_{p_0,p_0}}<1,
\end{equation}
then 
there exists a $v\in X_{p_0,\delta}(T)$, which solves the Navier Stokes equations (\ref{directsystem})-(\ref{directic}) in the sense of distributions and satisfies the following properties.
The first property is that $v$ solves the following integral equation:
\begin{equation}\label{vintegeqn}
v(x,t):= S(t)u_{0}+G(v\otimes v)(x,t)
\end{equation}
in $Q_{T}$, along with the estimate
\begin{equation}\label{vintegest}
\|v\|_{\X}<2\|S(t)u_0\|_{\X}:=2M^{(0)} .
\end{equation}
 The second property is (recalling that by assumption $u_0\in J(\mathbb{R}^3)$): 
\begin{equation}\label{venergyspace}
v\in\,W^{1,0}_{2}(Q_T)\cap C([0,T];J(\mathbb{R}^3))\cap L_{4}(Q_T)
.\end{equation}
Moreover, the following estimate is valid for $0<t\leq T$:
\begin{equation}\label{energyinequalityfluctuation}
\|G(v\otimes v)(\cdot,t)\|_{L_{2}}^2+\int\limits_0^t\int\limits_{\mathbb{R}^3} |\nabla G(v\otimes v)|^2dyd\tau\leq $$$$\leq C t^{\frac{1}{p_0-2}+2\delta\theta_{0}}(2M^{(0)})^{4\theta_0}\left(C t^{\frac{1}{2\theta_{0}-1} \left(\frac{1}{p_0-2}+2\delta\theta_0\right)}(2M^{(0)})^{\frac{4\theta_0}{2\theta_0-1}}+\|u_0\|_{L_2}^2\right)^{2(1-\theta_0)}.
\end{equation}
Here, 
\begin{equation}\label{thetadef}
\frac 1 4=\frac{\theta_0}{p_0}+\frac{1-\theta_0}{2}
\end{equation}
and $C=C(p_0,\delta,\theta_0)$.\\
 If $\pi_{v\otimes v}$ is the associated pressure we have (here, $\lambda\in ]0,T[$):
\begin{equation}\label{presspace}
\pi_{v\otimes v} \in L_{2}(Q_{T})\cap L_{\frac{p_0}{2},\infty}(Q_{\lambda,T})).
\end{equation}
The final property is that for $\lambda\in]0,T[$ and $k=0,1\ldots$, $l=0,1\ldots$:
\begin{equation}\label{vpsmooth}
\sup_{(x,t)\in Q_{\lambda,T}}|\partial_{t}^{l}\nabla^{k} v|+|\partial_{t}^{l}\nabla^{k} \pi_{v\otimes v}|\leq c(p_0,\delta,\lambda,\|u_0\|_{B_{p_0,p_0}^{s_{p_0}+\delta}},k,l ).
\end{equation}

\end{theorem}
\begin{proof}
Recall Young's inequality:
\begin{equation}\label{Youngs}
\|f\star g\|_{L^{r}}\leq C_{p,q}\|f\|_{L^{p}}\|g\|_{L^{q}}
\end{equation}
where $1<p,q,r<\infty$, $0<s_1,s_2\leq\infty$, $\frac 1 p+\frac 1 q=\frac 1 r +1$.

Applying this and the pointwise estimate (\ref{kenrelKest})  gives the following:
\begin{equation}\label{Ktensormainest1}
\|K(\cdot, t-\tau)\star (r\otimes r)(\cdot,\tau)\|_{L_{p_0}}\leq  \|K(\cdot,t-\tau)\|_{L_{{(p_0)}^{'}}}\|r\otimes r\|_{L_{\frac{p_0}{2}}}\leq$$$$\leq C (t-\tau)^{-(1+\frac{s_{p_0}}{2})}\|r\|_{L_{p_0}}^2\leq
C (t-\tau)^{-(1+\frac{s_{p_0}}{2})}\frac{\|r\|_{\X}^2}{\tau^{-s_{p_0}-\delta}}.
\end{equation}
 One can then show that
\begin{equation}\label{Gtensorest1}
\|G(r\otimes r)\|_{\X}\leq CT^{\frac{\delta}{2}} \|r\|_{\X}^2.
\end{equation}

We briefly describe successive approximations.
 For $n=1,2,\cdots$ let $v^{(0)}= S(t)u_0$, $$v^{(n+1)}= v^{(0)}+G(v^{(n)},v^{(n)}).$$
 Moreover for $n=0,1,\ldots$ define:
 \begin{equation}\label{Mdef}
 M^{(n)}:=\|v^{(n)}\|_{\X}.
 \end{equation}
 Then using (\ref{Gtensorest1}) we have the following iterative relation:
 \begin{equation}\label{Miterative}
 M^{(n+1)}\leq M^{(0)}+CT^{\frac{\delta}{2}}(M^{(n)})^2.
 \end{equation}
 
 If 
 \begin{equation}\label{Msmallnesscondition}
 4CT^{\frac{\delta}{2}}M^{(0)}< 1,
 \end{equation}
 then one can show that for $n=1,2\cdots$ we have
 \begin{equation}\label{Mbound}
 M^{(n)}<2M^{(0)}.
 \end{equation}
 \begin{itemize}
\item[] \textbf{Step 2: establishing energy bounds}
\end{itemize} First we note that by interpolation ($0\leq \tau\leq T$):
 \begin{equation}\label{interpolationineq}
 \|r(\cdot,\tau)\|_{L_{4}}\leq \|r(\cdot,\tau)\|_{L_{p_0}}^{\theta_0}\|r(\cdot,\tau)\|_{L_{2}}^{1-\theta_0}\leq  \|r\|_{\X}^{\theta_0}\|r(\cdot,\tau)\|_{L_{2}}^{1-\theta_0}\tau^{\theta_0(\frac{s_p}{2}+\frac{\delta}{2})}.
 \end{equation}
 Recall, 
 \begin{equation}\label{theta0recap}
 \frac{\theta_0}{p_0}+\frac{1-\theta_0}{2}=\frac{1}{4}.
 \end{equation}
 Specifically, 
 \begin{equation}\label{theta0ident}
 \theta_0=\frac{1/4}{1/2-1/p_0}.
 \end{equation}
 
 It is then immediate that
 $$\frac{-\theta_0 s_{p_0}}{2}= \frac{1-3/p_0}{4(1-2/p_0)}=\frac{1}{4}-\frac{1}{4(p_0-2)}<\frac{1}{4}.$$
 From this, we conclude that for $0\leq t\leq T$:
 \begin{equation}\label{L4spacetimeest}
 \|r\|_{L_{4}(Q_t)}\leq C t^{\frac{\theta_{0}\delta}{2}+\frac{1}{4(p_0-2)}}\|r\|_{\X}^{\theta_0}\|r\|_{L_{2,\infty}(Q_t)}^{1-\theta_0}.
 \end{equation}

 Let $r\in L_{4}(Q_T)\cap \X\cap L_{2,\infty}(Q_T)$ and $R:= G(r\otimes r)$. Furthermore, define $\pi_{r\otimes r}:= \mathcal{R}_{i}\mathcal{R}_{j}(r_{i}r_{j})$, where $\mathcal{R}_{i}$ denotes the Riesz transform and repeated indices are summed. 
 One can readily show that on $Q_{T}$,   $(R, \pi_{r\otimes r})$ are solutions to
 \begin{equation}\label{Reqn}
 \partial_{t} R-\Delta R+\rm{div}\,r\otimes r=-\nabla \pi_{r\otimes r},
 \end{equation}
 \begin{equation}\label{Rdivfree}
 \rm{div\,R}\,=0,
 \end{equation}
 \begin{equation}\label{Rintialcondition}
 R(\cdot,0)=0.
 \end{equation}
 We can also infer that 
 $R\in W^{1,0}_{2}(Q_T)\cap C([0,T]; J(\mathbb{R}^3))$, along with the estimate 
 \begin{equation}\label{Renergyest}
\|R(\cdot,t)\|_{L_{2}}^2+\int\limits_0^t\int\limits_{\mathbb{R}^3}|\nabla R(x,s)|^2 dxds\leq \|r\otimes r\|_{L_{2}(Q_t)}^2\leq$$$$\leq c\|r\|_{L_{4}(Q_{t})}^4  \leq C t^{2\theta_{0}\delta+\frac{1}{(p_0-2)}}\|r\|_{\X}^{4\theta_0}\|r\|_{L_{2,\infty}(Q_t)}^{4(1-\theta_0)}.
 \end{equation}
  Since the associated pressure is a composition of Riesz transforms acting on $r\otimes r$, we have the estimates
 \begin{equation}\label{presest1} 
 \|\pi_{r\otimes r}\|_{L_{2}(Q_T)}\leq C\|r\|_{L_{4}(Q_T)}^2,
 \end{equation}
 \begin{equation}\label{presest2}
 \|\pi_{r\otimes r}\|_{L_{\frac{p_0}{2}, \infty}(Q_{\lambda,T})}\leq  C(\lambda,T,p_0,\delta)\|r\|_{\X}^2.
 \end{equation}
 Moreover for $n=0,1,\ldots$ and $t\in [0,T]$ define:
 \begin{equation}\label{Edef}
 E^{(n)}(t):=\|v^{(n)}\|_{L_{\infty}(0,t;L_2)}+\|\nabla v^{(n)}\|_{L_{2}(Q_t)}.
 \end{equation}
 Clearly, by the assumptions on the initial data $E^{(0)}(t)\leq \|u_0\|_{L_2}<\infty$.
 Then from 
 (\ref{Renergyest}), we have the following iterative relations:
 \begin{equation}\label{Eiterative1}
 E^{(n+1)}(t)\leq E^{(0)}(t)+Ct^{\theta_{0}\delta+\frac{1}{2(p_0-2)}}(M^{(n)})^{2\theta_0}(E^{(n)}(t))^{2(1-\theta_0)}.
\end{equation}
 From (\ref{thetadef}), it is clear that $2(1-\theta_0)<1$.
 Hence by Young's inequality:
 \begin{equation}\label{Eiterative2}
 E^{(n+1)}(t)\leq Ct^{\frac{1}{2\theta_0-1}(\theta_{0}\delta+\frac{1}{2(p_0-2)})}(M^{(n)})^{\frac{2\theta_0}{2\theta_0-1}}+ E^{(0)}+\frac{1}{2}E^{(n)}(t).
\end{equation}
From (\ref{Mbound}) and iterations, we infer that for $t\in [0,T]$:
\begin{equation}\label{Ebounded}
E^{(n+1)}(t)\leq 2Ct^{\frac{1}{2\theta_0-1}(\theta_{0}\delta+\frac{1}{2(p_0-2)})}(2M^{(0)})^{\frac{2\theta_0}{2\theta_0-1}}+ 2E^{(0)}(t)\leq$$$$\leq 2Ct^{\frac{1}{2\theta_0-1}(\theta_{0}\delta+\frac{1}{2(p_0-2)})}(2M^{(0)})^{\frac{2\theta_0}{2\theta_0-1}}+ 2\|u_0\|_{L_{2}}
.
\end{equation}

 \begin{itemize}
 \item[]\textbf{Step 3: convergence and showing energy inequalities}
 \end{itemize}
 Using (\ref{Miterative})-(\ref{Mbound}), one can argue
 along the same lines as in \cite{Kato} to deduce that there exists $v\in \X$ such that
 \begin{equation}\label{v^nconverg1}
 \lim_{n\rightarrow 0}\|v^{(n)}-v\|_{\X}=0.
 \end{equation}
 Furthermore $v$ satisfies the integral equation (\ref{vintegeqn}) as well as satisfying the Navier Stokes equations, in the sense of distributions. From  $u_0\in J(\mathbb{R}^3)\cap \X$, (\ref{Ebounded}), (\ref{v^nconverg1}) and estimates analogous to (\ref{L4spacetimeest}) and (\ref{Renergyest}), applied to $v^{(m)}-v^{(n)}$, we have the following:
 \begin{equation}\label{v^nstrongconverg1}
 v^{(n)}\rightarrow v\,\rm{in}\,\, C([0,T];J(\mathbb{R}^3))\cap W^{1,0}_{2}(Q_T), 
 \end{equation}
 \begin{equation}\label{v^ninitialcondition}
 v(\cdot,0)=u_0,
 \end{equation}
 \begin{equation}\label{v^nstrongconverg2}
 v^{(n)}\rightarrow v\,\rm{in}\,\, L_{4}(Q_T),
 \end{equation}
 \begin{equation}\label{presconvergence}
 \pi_{v^{(n+1)}\otimes v^{(n+1)}}\rightarrow \pi_{v\otimes v}\,\rm{in}\, L_{2}(Q_T) \,\,\rm{and}\,\, L_{\frac{p_0}{2},\infty}(Q_{\lambda,T})
 \end{equation}
 where $0<\lambda<T$. Using this, along with (\ref{Renergyest}) and (\ref{Ebounded}), we infer the estimate (\ref{energyinequalityfluctuation}). 
\begin{itemize}
 \item[]\textbf{Step 4: estimate of higher derivatives}
 \end{itemize}
 All that remains to prove is the estimate (\ref{vpsmooth}).
 First note that from the definition of $\X$, we have the estimate:
 \begin{equation}\label{limitvelocityest}
 \|v\|_{L_{p_0,\infty}(Q_{\lambda,T})}\leq \lambda^{\frac 1 2(s_{p_0}+\delta)}\|v\|_{\X}.
\end{equation}
Since $\pi_{v\otimes v}$ is a convolution of Riesz transforms, we deduce from (\ref{limitvelocityest}) that
\begin{equation}\label{limitpressureest}
 \|\pi_{v\otimes v}\|_{L_{\frac{p_0}{2},\infty}(Q_{\lambda,T})}\leq C \lambda^{s_{p_0}+\delta}\|v\|_{\X}^2.
 \end{equation}
 One can infer that $(v,\pi_{v\otimes v})$ satisfies the local energy equality. This can be shown using (\ref{venergyspace})-(\ref{presspace}) and a mollification argument. If $(x,t)\in Q_{\lambda,T}$, then for $0<r^2<\frac{\lambda}{2}$ we can apply H\"{o}lder's inequality and (\ref{limitvelocityest})-(\ref{limitpressureest}) to infer 
 \begin{equation}\label{CKNquantity}
 \frac{1}{r^2}\int\limits_{t-r^2}^{t}\int\limits_{B(x,r)}(|v|^3+|\pi_{v\otimes v}|^{\frac{3}{2}}) dxdt\leq  C\lambda^{\frac{3}{2}(s_{p_0}+\delta)} r^{3(1-\frac{3}{p_0})}\|u_0\|_{\dot{B}_{p_0,p_0}^{s_{p_0}+\delta}}^3.
 \end{equation}
 Clearly, there exists $r_0^2(\lambda, \varepsilon_{CKN}, \|u_0\|_{\dot{B}_{p_0,p_0}^{s_{p_0}+\delta}})<\frac{\lambda}{2}$ such that
 $$\frac{1}{r_0^2}\int\limits_{t-r_0^2}^{t}\int\limits_{B(x,r_0)}(|v|^3+|\pi_{v\otimes v}|^{\frac{3}{2}}) dxdt\leq \varepsilon_{CKN}.$$
 By the $\varepsilon$- regularity theory developed in \cite{CKN}, there exists universal constants $c_{0k}>0$ such that (for $(x,t)$ and $r$ as above) we have
$$ |\nabla^{k} v(x,t)|\leq \frac{c_{0k}}{r_{0}^{k+1}}= C(k, \lambda, \|u_0\|_{\dot{B}_{p_0,p_0}^{s_{p_0}+\delta}}).$$
Thus, 
\begin{equation}\label{vhigherreg}
\sup_{(x,t)\in Q_{\lambda,T}} |\nabla^k v|\leq C(k, \lambda, \|u_0\|_{\dot{B}_{p_0,p_0}^{s_{p_0}+\delta}}).
\end{equation}
Using this and (\ref{limitpressureest}), we obtain by local regularity theory for elliptic equations that
\begin{equation}\label{preshigherreg}
\sup_{(x,t)\in Q_{\lambda,T}} |\nabla^k \pi_{v\otimes v}|\leq C(k, \lambda, \|u_0\|_{\dot{B}_{p_0,p_0}^{s_{p_0}+\delta}}).
\end{equation}
 From these estimates, the singular integral representation of $\pi_{v\otimes v}$ and that $(v,\pi_{v\otimes v})$ satisfy the Navier-Stokes system  one can prove the corresponding estimates hold for higher time derivatives of the velocity field and pressure.

 \end{proof}

\subsection{Proof of Lemma \ref{estnearinitialforLeraywithbesov}}
The  proof of Lemma \ref{estnearinitialforLeraywithbesov} is achieved by careful analysis of  certain decompositions of weak Leray-Hopf solutions, with initial data in the class given in Lemma \ref{estnearinitialforLeraywithbesov}. A key part of this involves decompositions of the initial data (Corollary \ref{Decomp}), together with properties of mild solutions, whose initial data belongs in a subcritical homogeneous Besov space (Theorem \ref{regularity}). In the context of local energy solutions of the Navier-Stokes equations with $L_{3}$ initial data, related splitting arguments have been used in \cite{jiasverak}.

Before proceeding, we state a known lemma found in \cite{prodi} and \cite{Serrin}, for example.
\begin{lemma}\label{trilinear}
Let $p\in ]3,\infty]$ and 
\begin{equation}\label{serrinpairs}
\frac{3}{p}+\frac{2}{r}=1.
\end{equation}
Suppose that $w \in L_{p,r}(Q_T)$, $v\in L_{2,\infty}(Q_T)$ and $\nabla v\in L_{2}(Q_T)$. 
Then for $t\in ]0,T[$:
\begin{equation}\label{continuitytrilinear}
\int\limits_0^t\int\limits_{\mathbb{R}^3} |\nabla v||v||w| dxdt'\leq C\int\limits_0^t \|w\|_{L_{p}}^r \|v\|^{2}_{L_2}dt'+\frac{1}{2}\int\limits_0^t\int\limits_{\mathbb{R}^3} |\nabla v|^2 dxdt'.
\end{equation}
\end{lemma}

\begin{proof}
Throughout this subsection, $u_0 \in \dot{B}_{p,p}^{s_{p,\alpha}}(\mathbb{R}^3)\cap L_2(\mathbb{R}^3)$. Here, $p$ and $\alpha$ satisfy the assumptions of Theorem \ref{weakstronguniquenessBesov}. We will write $u_0= \bar{u}^{1, \epsilon}+\bar{u}^{2, \epsilon}.$
 Here the decomposition has been performed according to Corollary \ref{Decomp} (specifically, (\ref{udecomp1})-(\ref{baru_1est.1})), with $\epsilon>0.$ Thus,
 \begin{equation}\label{baru_1u_0decomp}
 \|\bar{u}^{1, \epsilon}\|_{\dot{B}^{s_{p_0}+\delta}_{p_0,p_0}}^{p_0}\leq \epsilon^{p_0-p}\|u_0\|_{\dot{B}^{s_{p,\alpha}}_{p,p}}^{p}
 \end{equation}
 \begin{equation}\label{baru_2u_0decomp}
 \|\bar{u}^{2, \epsilon}\|_{L_2}^2\leq C(p,p_0,\|\mathcal{F}^{-1}\varphi\|_{L_1})\epsilon^{2-p}\|u_0\|_{\dot{B}^{s_{p,\alpha}}_{p,p}}^{p}
 \end{equation}
 and
 \begin{equation}\label{baru_1u_0decomp.1}
 \|\bar{u}^{1, \epsilon}\|_{L_{2}}\leq C(\|\mathcal{F}^{-1}\varphi\|_{L_1})\|u_0\|_{L_2}.
 \end{equation}
 
  Throughout this section we will let  $w^{\epsilon}$ be the mild solution from Theorem \ref{regularity} generated by the initial data $\bar{u}^{1, \epsilon}$. Recall from Theorem \ref{regularity}
that $w^{\epsilon}$ is defined on $Q_{T_{\epsilon}}$,
where
\begin{equation}\label{smallnessasmprecall}
4c(p_0)T_{\epsilon}^{\frac{\delta}{2}}\|\bar{u}^{1,\epsilon}\|_{\dot{B}^{s_{p_0}+\delta}_{p_0,p_0}}<1.
\end{equation}
In accordance with this and (\ref{baru_1u_0decomp}), we will take
\begin{equation}\label{Tepsilondef}
T_{\epsilon}:=\frac{1}{\Big(8c(p_0)\epsilon^{\frac{p_0-p}{p_0}}\|u_0\|_{\dot{B}^{s_{p,\alpha}}_{p,p}}^{\frac{p}{p_0}}\Big)^{\frac{2}{\delta}}}
\end{equation}
The two main estimates we will use are as follows. Using (\ref{Besovembedding}), (\ref{vintegest}) and (\ref{baru_1u_0decomp}), we have
\begin{equation}\label{wepsilonkatoest}
\|w^{\epsilon}\|_{X_{p_0,\delta}(T_{\epsilon})}<C\epsilon^{\frac{p_0-p}{p_0}}\|u_0\|_{\dot{B}^{s_{p,\alpha}}_{p,p}}^{\frac{p}{p_0}}.
\end{equation}
 The second property, from Theorem \ref{regularity}, is (recalling that by assumption $\bar{u}^{1,\epsilon}\in L_{2}(\mathbb{R}^3)$): 
\begin{equation}\label{wepsilonenergyspace}
w^{\epsilon}\in\,W^{1,0}_{2}(Q_{T_{\epsilon}})\cap C([0,T_{\epsilon}];J(\mathbb{R}^3))\cap L_{4}(Q_{T_{\epsilon}})
.\end{equation}
Consequently, it can be shown that $w^{\epsilon}$ satisfies the energy equality:
\begin{equation}\label{energyequalitywepsilon}
\|w^{\epsilon}(\cdot,s)\|_{L_2}^2+2\int\limits_{\tau}^{s}\int\limits_{\mathbb{R}^3} |\nabla w^{\epsilon}(y,\tau)|^2 dyd\tau=\|w^{\epsilon}(\cdot,s')\|_{L_2}^2
\end{equation}
for $0\leq s'\leq s\leq T_{\epsilon}$.
Moreover,  using (\ref{energyinequalityfluctuation}), (\ref{baru_1u_0decomp}) and (\ref{baru_1u_0decomp.1}), the following estimate is valid for $0\leq s \leq T_{\epsilon}$:
\begin{equation}\label{energyinequalitywepsilonfluctuation}
\|w^{\epsilon}(\cdot,s)-S(s)\bar{u}^{1,\epsilon}\|_{L_{2}}^2+\int\limits_0^s\int\limits_{\mathbb{R}^3} |\nabla w^{\epsilon}(\cdot,\tau)-\nabla S(\tau)\bar{u}^{1,\epsilon}|^2dyd\tau\leq $$$$\leq C s^{\frac{1}{p_0-2}+2\delta\theta_{0}}(\epsilon^{\frac{p_0-p}{p_0}}\|u_0\|_{\dot{B}^{s_{p,\alpha}}_{p,p}}^{\frac{p}{p_0}})^{4\theta_0}\times$$$$\times\left( s^{\frac{1}{2\theta_{0}-1} \left(\frac{1}{p_0-2}+2\delta\theta_0\right)}(\epsilon^{\frac{p_0-p}{p_0}}\|u_0\|_{\dot{B}^{s_{p,\alpha}}_{p,p}}^{\frac{p}{p_0}})^{\frac{4\theta_0}{2\theta_0-1}}+\|u_0\|_{L_2}^2\right)^{2(1-\theta_0)}.
\end{equation}
Here, 
\begin{equation}\label{thetadefrecall}
\frac 1 4=\frac{\theta_0}{p_0}+\frac{1-\theta_0}{2}
\end{equation}
and $C=C(p_0,\delta,\theta_0)$.
Let $$S_{\epsilon}:=\min(T,T_{\epsilon}).$$
Define on $Q_{S_{\epsilon}}$, $v^{\epsilon}(x,t):= u(x,t)- w^{\epsilon}(x,t)$. Clearly we have :
\begin{equation} \label{vepsilonspaces}
v^{\epsilon}\in L_{\infty}(0,S_{\epsilon}; L_{2})\cap C_{w}([0,S_{\epsilon}]; J(\mathbb{R}^3))\cap W^{1,0}_{2}(Q_{S_{\epsilon}})
\end{equation}
and
\begin{equation} \label{vepsiloncontinuityattzero}
\lim_{\tau\rightarrow 0}\|v^{\epsilon}(\cdot,\tau)-\bar{u}^{2,\epsilon}\|_{L_{2}}=0.
\end{equation}
Moreover, $v^{\epsilon}$ satisfies the following equations
\begin{equation}\label{directsystemvepsilon}
\partial_t v^{\epsilon}+v^{\epsilon}\cdot\nabla v^{\epsilon}+w^{\epsilon}\cdot\nabla v^{\epsilon}+v^{\epsilon}\cdot\nabla w^{\epsilon}-\Delta v^{\epsilon}=-\nabla p^{\epsilon},\qquad\mbox{div}\,v^{\epsilon}=0
\end{equation}
in $Q_{S_{\epsilon}}$,
with the initial conditions
\begin{equation}\label{directicvepsilon}
v^{\epsilon}(\cdot,0)=\bar{u}^{2,\epsilon}(\cdot)
\end{equation}
in $\mathbb{R}^3$.
From (\ref{wepsilonkatoest}), the definition of ${X_{p_0,\delta}(T_{\epsilon})}$ and (\ref{baru_1u_0decomp}), we see that for $0<s\leq T_{\epsilon}$:
\begin{equation}\label{serrinnormvepsilon}
\int\limits_{0}^{s} \|w^{\epsilon}(\cdot,\tau)\|_{L_{p_0}}^{r_0} d\tau\leq C(\delta,p_0) s^{\frac{\delta r_0}{2}} ( \epsilon^{\frac{p_0-p}{p_0}}\|u_0\|_{\dot{B}^{s_{p,\alpha}}_{p,p}}^{\frac{p}{p_0}})^{{r_0}}.
\end{equation}
Here, $r_0\in ]2,\infty[$ is such that $$\frac{3}{p_0}+\frac{2}{r_0}=1.$$

Moreover, the following energy inequality holds for $s\in [0,S_{\epsilon}]$ \footnote{ This can be justified by applying Proposition 14.3 in \cite{LR1}, for example.}:
\begin{equation}\label{venergyequality}
\|v^{\epsilon}(\cdot,s)\|_{L_2}^2+ 2\int\limits_0^s\int\limits_{\mathbb{R}^3} |\nabla v^{\epsilon}(x,\tau)|^2 dxd\tau \leq \|\bar{u}^{2,\epsilon}\|_{L_2}^2+$$$$+2\int\limits_0^s\int\limits_{\mathbb{R}^3}v^{\epsilon}\otimes w^{\varepsilon}:\nabla v^{\epsilon}  dxd\tau.
\end{equation}
We may then apply Lemma \ref{trilinear} to infer that for $s\in [0,S_{\epsilon}]$:
\begin{equation}\label{venergyinequalitynorms}
\|v^{\epsilon}(\cdot,s)\|_{L_2}^2+ \int\limits_0^s\int\limits_{\mathbb{R}^3} |\nabla v^{\epsilon}(x,\tau)|^2 dxd\tau\leq \|\bar{u}^{2, \epsilon}\|_{L_2}^2+$$$$+C\int\limits_0^s \|v^{\epsilon}(\cdot,\tau)\|_{L_2}^2 \|w^{\varepsilon}(\cdot,\tau)\|_{L_{p_0}}^{r_0} d\tau.
\end{equation}
By an application of the Gronwall lemma, we see that for $s\in [0,S_{\epsilon}]$:
\begin{equation}\label{gronwallresult}
\|v^{\epsilon}(\cdot,s)\|_{L_2}^2\leq 
  C\|\bar{u}^{2,\epsilon}\|_{L_2}^2 \exp{\Big(C\int\limits_{0}^s \|w^{\epsilon}(\cdot,\tau)\|_{L_{p_0}}^{r_0} d\tau\Big)}.
\end{equation}
We may then use (\ref{baru_2u_0decomp}) and (\ref{serrinnormvepsilon}) to infer for $s\in [0,S_{\epsilon}]$:
\begin{equation}\label{vepsilonkineticenergybound}
\|v^{\epsilon}(\cdot,s)\|_{L_2}^2\leq 
  C(p_0,p)\epsilon^{2-p}\|u_0\|_{\dot{B}^{s_{p,\alpha}}_{p,p}}^{p} \exp{\Big(C(\delta,p_0) s^{\frac{\delta r_0}{2}} \Big( \epsilon^{\frac{p_0-p}{p_0}}\|u_0\|_{\dot{B}^{s_{p,\alpha}}_{p,p}}^{\frac{p}{p_0}}\Big)^{r_0}\Big)}.
\end{equation}
Clearly, for $s\in [0,S_{\epsilon}]$ we have
$$\|u(\cdot,s)-S(s)u_0\|_{L_2}^2\leq C(\|v^{\epsilon}(\cdot,s)\|_{L_{2}}^2+\|\bar{u}^{2,\epsilon}\|_{L_2}^2 +\|w^{\epsilon}(\cdot,s)-S(s)\bar{u}^{1,\epsilon}\|_{L_{2}}^2).$$
Thus, using this together with (\ref{baru_2u_0decomp}), (\ref{energyinequalitywepsilonfluctuation}) and (\ref{vepsilonkineticenergybound}), we infer that for $s\in [0,S_{\epsilon}]$:
\begin{equation}\label{estnearinitialtimeepsilon}
\|u(\cdot,s)-S(s)u_0\|_{L_2}^2\leq$$$$\leq C(p_0,p)\epsilon^{2-p}\|u_0\|_{\dot{B}^{s_{p,\alpha}}_{p,p}}^{p} \Big(\exp{\Big(C(\delta,p_0) s^{\frac{\delta r_0}{2}} \Big( \epsilon^{\frac{p_0-p}{p_0}}\|u_0\|_{\dot{B}^{s_{p,\alpha}}_{p,p}}^{\frac{p}{p_0}}\Big)^{r_0}\Big)}+1\Big)+$$$$+\Big(C s^{\frac{1}{p_0-2}+2\delta\theta_{0}}\Big(\epsilon^{\frac{p_0-p}{p_0}}\|u_0\|_{\dot{B}^{s_{p,\alpha}}_{p,p}}^{\frac{p}{p_0}}\Big)^{4\theta_0}\times$$$$\times \left( s^{\frac{1}{2\theta_{0}-1} \left(\frac{1}{p_0-2}+2\delta\theta_0\right)}\Big(\epsilon^{\frac{p_0-p}{p_0}}\|u_0\|_{\dot{B}^{s_{p,\alpha}}_{p,p}}^{\frac{p}{p_0}}\Big)^{\frac{4\theta_0}{2\theta_0-1}}+\|u_0\|_{L_2}^2\right)^{2(1-\theta_0)}\Big).
\end{equation}
 Take $\epsilon= t^{-\gamma}$, where  $0<t\leq\min(1,T)$ and $\gamma>0$. Observing (\ref{Tepsilondef}), we see that in order to also satisfy $t\in ]0,T_{\epsilon}]$ (and hence $t\in ]0,\min(1,S_{\epsilon})]$) , we  should take
\begin{equation}\label{gammarequirement1}
0<\gamma<\frac{\delta p_0}{2(p_0-p)}
\end{equation}
and we should consider the additional restriction 
\begin{equation}\label{trequirement}
0<t<\min\Big(1,T,\Big( \frac{1}{8c} \|u_0\|_{\dot{B}^{s_{p,\alpha}}_{p,p}}^{-\frac{p}{p_0}}\Big)^{\frac{2p_0}{\delta p_0-2\gamma(p_0-p)}}\Big).
\end{equation}
Assuming these restrictions, we see that according to (\ref{estnearinitialtimeepsilon}), we then have: 
\begin{equation}\label{estnearinitialtimeepsilon.1}
\|u(\cdot,t)-S(t)u_0\|_{L_2}^2\leq$$$$\leq C(p_0,p)t^{\gamma(p-2)}\|u_0\|_{\dot{B}^{s_{p,\alpha}}_{p,p}}^{p} \Big(\exp{\Big(C(\delta,p_0) t^{\frac{\delta r_0}{2}} \Big( t^{\frac{-\gamma(p_0-p)}{p_0}}\|u_0\|_{\dot{B}^{s_{p,\alpha}}_{p,p}}^{\frac{p}{p_0}}\Big)^{r_0}\Big)}+1\Big)+$$$$+\Big(C t^{\frac{1}{p_0-2}+2\delta\theta_{0}}\Big(t^{\frac{-\gamma(p_0-p)}{p_0}}\|u_0\|_{\dot{B}^{s_{p,\alpha}}_{p,p}}^{\frac{p}{p_0}}\Big)^{4\theta_0}\times$$$$\times\left( t^{\frac{1}{2\theta_{0}-1} \left(\frac{1}{p_0-2}+2\delta\theta_0\right)}\Big(t^{\frac{-\gamma(p_0-p)}{p_0}}\|u_0\|_{\dot{B}^{s_{p,\alpha}}_{p,p}}^{\frac{p}{p_0}}\Big)^{\frac{4\theta_0}{2\theta_0-1}}+\|u_0\|_{L_2}^2\right)^{2(1-\theta_0)}\Big).
\end{equation} 
Let us further choose $\gamma$ such that the following inequalities hold:
\begin{equation}\label{gammarequirement2}
\frac{\delta r_0}{2}-\frac{\gamma r_0(p_0-p)}{p_0}>0,
\end{equation}
\begin{equation}\label{gammarequirement3}
\frac{1}{p_0-2}+2\delta\theta_0-\frac{4\gamma\theta_0(p_0-p)}{p_0}>0
\end{equation}
and
\begin{equation}\label{gammarequirement4}
\frac{1}{2\theta_0-1}\Big(\frac{1}{p_0-2}+2\delta\theta_0\Big)-\frac{4\theta_0\gamma(p_0-p)}{p_0(2\theta_0-1)}>0.
\end{equation}
With these choices, together with (\ref{gammarequirement1})-(\ref{trequirement}), we recover the conclusions of Lemma \ref{estnearinitialforLeraywithbesov}.
\end{proof}

\setcounter{equation}{0}
\section{ Short Time Uniqueness of Weak Leray-Hopf Solutions for Initial Values in $VMO^{-1}$}
\subsection{Construction of Strong Solutions}
The approach we will take to prove Theorem \ref{weakstronguniquenessBesov} is as follows. Namely, we construct a weak Leray-Hopf solution, with initial data $u_0 \in J(\mathbb{R}^3)\cap \dot{B}^{s_{p,\alpha}}_{p,p}(\mathbb{R}^3)\cap VMO^{-1}(\mathbb{R}^3)$, by perturbation methods. We refer to this constructed solution as the 'strong solution'. Then, Lemma \ref{estnearinitialforLeraywithbesov} plays a crucial role in showing that the strong solution has good enough properties to coincide with all weak Leray-Hopf solutions, with the same initial data, on some  time interval $]0,T(u_0)[$. With this in mind, we now state the relevant Theorem related to the construction of this 'strong solution'. 
Let us introduce the necessary preliminaries.
The path space $\mathcal{X}_{T}$ for the mild solutions constructed in \cite{kochtataru} is defined to be 
\begin{equation}\label{pathspaceBMO-1}
\mathcal{X}_{T}:=\{ u\in \mathcal{S}^{'}(\mathbb{R}^3\times \mathbb{R}_{+}): \|e^{t\Delta}v\|_{\mathcal{E}_{T}}<\infty\}.
\end{equation}
Here,
\begin{equation}\label{pathspacenormdef}
\|u\|_{\mathcal{E}_{T}}:= \sup_{0<t<T} \sqrt{t}\|u(\cdot,t)\|_{L_{\infty}(\mathbb{R}^3)}+$$$$+\sup_{(x,t)\in \mathbb{R}^3\times ]0,T[}\Big(\frac{1}{|B(0,t)|}\int\limits_0^t\int\limits_{|y-x|<\sqrt{t}} |u|^2 dyds\Big)^{\frac{1}{2}}.
\end{equation}
From (\ref{bmo-1norm}), we see that for $0<T\leq\infty$
 \begin{equation}\label{BMO-1embeddingcritical}
u_0 \in BMO^{-1}(\mathbb{R}^3)\Rightarrow \|S(t)u_{0}\|_{\mathcal{E}_{T}}\leq C\|u_0\|_{BMO^{-1}}.
\end{equation}
Since $C_{0}^{\infty}(\mathbb{R}^3)$ is dense in $VMO^{-1}(\mathbb{R}^3)$, we can see from the above that for $u_0\in VMO^{-1}(\mathbb{R}^3)$
\begin{equation}\label{VMO-1shrinking}
\lim_{T\rightarrow 0^{+}} \|S(t)u_{0}\|_{\mathcal{E}_{T}}=0.
\end{equation}
Recalling the definition of $G(f\otimes g)$ given by (\ref{fluctuationdef}), it was shown in \cite{kochtataru} that there exists a universal constant $C$ such that for all $f,\,g\in \mathcal{E}_{T}$ 
\begin{equation}\label{bilinbmo-1}
\|G(f\otimes g)\|_{\mathcal{E}_{T}}\leq C\|f\|_{\mathcal{E}_{T}}\|g\|_{\mathcal{E}_{T}}.
\end{equation}
Here is the Theorem related to the construction of the 'strong solution'. The main features of this construction required for our purposes, can already be inferred ideas contained in \cite{LRprioux}.
Since the proof is not explicitly contained in \cite{LRprioux}, we find it beneficial to sketch certain parts of the proof in the Appendix.
\begin{theorem}\label{regularitycriticalbmo-1} 
Suppose that $u_0\in VMO^{-1}(\mathbb{R}^3)\cap J(\mathbb{R}^3).$ 
There exists a  universal constant $\epsilon_0>0$ such that 
if 
\begin{equation}\label{smallnessasmpbmo-1}
\|S(t)u_{0}\|_{\mathcal{E}_{T}}<\epsilon_0,
\end{equation}
then 
there exists a $v\in \mathcal{E}_{T}$, which solves the Navier Stokes equations in the sense of distributions and satisfies the following properties.
The first property is that $v$ solves the following integral equation:
\begin{equation}\label{vintegeqnbmo-1}
v(x,t):= S(t)u_{0}+G(v\otimes v)(x,t)
\end{equation}
in $Q_{T}$, along with the estimate
\begin{equation}\label{vintegestbmo-1}
\|v\|_{\mathcal{E}_{T}}<2\|S(t)u_0\|_{\mathcal{E}(T)}.
\end{equation}
 The second property
  is that $v$ is a weak Leray-Hopf solution on $Q_{T}$. 

 If $\pi_{v\otimes v}$ is the associated pressure we have (here, $\lambda\in ]0,T[$ and $p\in ]1,\infty[$):
\begin{equation}\label{presspacebmo-1}
\pi_{v\otimes v} \in L_{\frac{5}{3}}(Q_{T})\cap L_{\frac{p}{2},\infty}(Q_{\lambda,T})).
\end{equation}
The final property is that for $\lambda\in]0,T[$ and $k=0,1\ldots$, $l=0,1\ldots$:
\begin{equation}\label{vpsmoothbmo-1}
\sup_{(x,t)\in Q_{\lambda,T}}|\partial_{t}^{l}\nabla^{k} v|+|\partial_{t}^{l}\nabla^{k} \pi_{v\otimes v}|\leq c(p_0,\lambda,\|u_0\|_{BMO^{-1}},\|u_0\|_{L_{2}},k,l ).
\end{equation}
\end{theorem}

\subsection{Proof of Theorem \ref{weakstronguniquenessBesov}}
\begin{proof}
Let us now consider any other weak Leray-Hopf solution $u$, defined on $Q_{\infty}$ and with initial data $u_0 \in J(\mathbb{R}^3)\cap \dot{B}^{s_{p,\alpha}}_{p,p}(\mathbb{R}^3)\cap VMO^{-1}(\mathbb{R}^3).$
Let $\widehat{T}(u_0)$ be such that
$$\|S(t)u_0\|_{\mathcal{E}_{\widehat{T}(u_0)}}< \epsilon_0,$$
where $\epsilon_0$ is from (\ref{smallnessasmpbmo-1}) of Theorem \ref{regularitycriticalbmo-1}. Consider $0<T<\widehat{T}(u_0)$, where $T$ is to be determined.
Let $v:Q_{T}\rightarrow\mathbb{R}^3$ be as in Theorem \ref{regularitycriticalbmo-1}. From (\ref{vintegestbmo-1}), we have
\begin{equation}\label{vbmo-1estuptoT}
\|v\|_{\mathcal{E}_{T}}<2\|S(t)u_0\|_{\mathcal{E}_{T}}\leq 2\|S(t)u_0\|_{\mathcal{E}_{\widehat{T}(u_0)}}<2\epsilon_0.
\end{equation}
We define 
\begin{equation}\label{wdefbmo-1}
w=u-v \in W^{1,0}_{2}(Q_{T})\cap C_{w}([0,T]; J(\mathbb{R}^3)).
\end{equation}
Moreover, $w$ satisfies the following equations
\begin{equation}\label{directsystemwbmo-1}
\partial_t w+w\cdot\nabla w+v\cdot\nabla w+w\cdot\nabla v-\Delta w=-\nabla q,\qquad\mbox{div}\,w=0
\end{equation}
in $Q_{T}$,
with the initial condition satisfied in the  strong $L_{2}$ sense:
\begin{equation}\label{initialconditionwbmo-1}
\lim_{t\rightarrow 0^{+}}\|w(\cdot,0)\|_{L_{2}}=0
\end{equation}
in $\mathbb{R}^3$.
From the definition of $\mathcal{E}_{T}$,
 we have that $v\in L_{\infty}(Q_{\delta,T})$ for $0<\delta<s\leq T$.


Using Proposition 14.3 in \cite{LR1}, one can deduce that for $t\in [\delta,T]$: 
\begin{equation}\label{wenergyequalitybmo-1}
\|w(\cdot,t)\|_{L_2}^2\leq \|u_{0}\|_{L_2}^2-2\int\limits_0^{t}\int\limits_{\mathbb{R}^3} |\nabla u|^2 dyd\tau+\|u_{0}\|_{L_2}^2-2\int\limits_0^{t}\int\limits_{\mathbb{R}^3} |\nabla v|^2 dyd\tau-$$$$-2\int\limits_{\mathbb{R}^3} u(y,\delta)\cdot v(y,\delta)dy+4\int_{\delta}^t\int\limits_{\mathbb{R}^3} \nabla u:\nabla v dyd\tau+$$$$+2\int\limits_{\delta}^t\int\limits_{\mathbb{R}^3} v\otimes w:\nabla w dyd\tau.
\end{equation}
Using Lemma \ref{trilinear} and (\ref{vintegestbmo-1}) , we see that
\begin{equation}\label{trilineatestwbmo-1}
\int\limits_{\delta}^t\int\limits_{\mathbb{R}^3} |v||w||\nabla w| dyd\tau\leq C\int\limits_{\delta}^t \|v(\cdot,\tau)\|_{L_{\infty}}^{2}\|w(\cdot,\tau)\|_{L_{2}}^2 d\tau+$$$$+\frac{1}{2}\int\limits_{\delta}^t\int\limits_{\mathbb{R}^3}|\nabla w|^2 dyd\tau \leq $$$$\leq C\|v\|_{\mathcal{E}_{T}}^{2}\int\limits_{\delta}^t \frac{\|w(\cdot,\tau)\|_{L_{2}}^{2}}{\tau} d\tau+ \frac{1}{2}\int\limits_{\delta}^t\int\limits_{\mathbb{R}^3}|\nabla w|^2 dyd\tau.
\end{equation}
The main point now is that Lemma \ref{estnearinitialforLeraywithbesov} implies that there exists $$\beta(p,\alpha)>0$$ and $$\gamma(\|u_{0}\|_{\dot{B}^{s_{p,\alpha}}_{p,p}(\mathbb{R}^3)}, p,\alpha)>0$$ such that for $0<t<\min (1,T, \gamma)$ :
\begin{equation}\label{estnearinitialtimeleraydifferenceslesbesgue}
\|w(\cdot,t)\|_{L_{2}}^2\leq t^{\beta} c(p,\alpha, \|u_0\|_{L_{2}(\mathbb{R}^3)}, \|u_0\|_{\dot{B}^{s_{p,\alpha}}_{p,p}(\mathbb{R}^3)}).
\end{equation}
Hence,
\begin{equation}\label{wweightedintimelesbesgue}
\sup_{0<t<T} \frac{\|w(\cdot,t)\|_{L_2}^2}{t^{\beta}}<\infty.
\end{equation}
This allows us to take $\delta\rightarrow 0$ in (\ref{wenergyequalitybmo-1}) to get:
\begin{equation}\label{wenergyequalityuptoinitialtimebmo-1lesbesgue}
\|w(\cdot,t)\|_{L_2}^2+2\int\limits_{0}^{t}\int\limits_{\mathbb{R}^3}|\nabla w|^2 dyd\tau \leq 2\int\limits_{0}^t\int\limits_{\mathbb{R}^3} v\otimes w:\nabla w dyd\tau.
\end{equation}
Using (\ref{trilineatestwbmo-1}) and (\ref{wweightedintimelesbesgue}) we see that for $t\in [0,T]$: 
\begin{equation}\label{wmainest1bmo-1}
\|w(\cdot,t)\|_{L_{2}}^2\leq C\|v\|_{\mathcal{E}_{T}}^{2}\int\limits_{0}^t \frac{\|w(\cdot,\tau)\|_{L_{2}}^{2}}{\tau} d\tau\leq $$$$\leq{\frac{C}{\beta}t^{\beta}\|v\|_{\mathcal{E}_{T}}^{2}}\sup_{0<\tau<T}\Big(\frac{\|w(\cdot,\tau)\|_{L_{2}}^2}{\tau^{\beta}}\Big).
\end{equation}
Using this and (\ref{vbmo-1estuptoT}), we have
\begin{equation}\label{wmainest2bmo-1}
\sup_{0<\tau<T}\Big(\frac{\|w(\cdot,\tau)\|_{L_{2}}^2}{\tau^{\beta}}\Big)\leq{\frac{C'}{\beta}\|S(t)u_0\|_{\mathcal{E}_{T}}^{2}}\sup_{0<\tau<T}\Big(\frac{\|w(\cdot,\tau)\|_{L_{2}}^2}{\tau^{\beta}}\Big).
\end{equation}
Using (\ref{VMO-1shrinking}), we see that we can choose $0<T=T(u_0)<\widehat{T}(u_0)$ such that $$\|S(t)u_0\|_{\mathcal{E}_{T}}\leq \min\Big(\frac{\beta}{2C'}\Big)^{\frac 1 2}.$$
With this choice of $T(u_0)$, it immediately follows that $w=0$ on $Q_{T(u_0)}$. 
\end{proof} 
\subsection{Proof of Corollary \ref{cannoneweakstronguniqueness}}
\begin{proof}
Since $3<p<\infty$, it is clear that there exists an $\alpha:=\alpha(p)$ such that $2<\alpha<3$ and
\begin{equation}\label{pcondition}
\alpha<p<\frac{\alpha}{3-\alpha}.
\end{equation}
With this $p$ and $\alpha$, we may apply Proposition \ref{interpolativeinequalitybahourichemindanchin} with
$s_{1}= -\frac{3}{2}+\frac{3}{p},$
$s_{2}= -1+\frac{3}{p}$
and $\theta=6\Big(\frac{1}{\alpha}-\frac{1}{3}\Big).$ In particular this gives for any $u_0\in \mathcal{S}^{'}_{h}$:
$$ 
\|u_0\|_{\dot{B}^{s_{p,\alpha}}_{p,1}}\leq c(p,\alpha)\|u_0\|_{\dot{B}^{-\frac{3}{2}+\frac{3}{p}}_{p,\infty}}^{6(\frac{1}{\alpha}-\frac{1}{3})}\|u_0\|_{\dot{B}^{-1+\frac{3}{p}}_{p,\infty}}^{6(\frac{1}{2}-\frac{1}{\alpha})}.
$$ 
From Remark \ref{besovremark2}, we see that $\dot{B}^{s_{p,\alpha}}_{p,1}(\mathbb{R}^3)\hookrightarrow \dot{B}^{s_{p,\alpha}}_{p,p}(\mathbb{R}^3) $ and
$L_{2}(\mathbb{R}^3) \hookrightarrow\dot{B}^{0}_{2,\infty}(\mathbb{R}^3)\hookrightarrow\dot{B}^{-\frac{3}{2}+\frac{3}{p}}_{p,\infty}(\mathbb{R}^3).$
Thus, we have the inclusion
$$
\dot{B}^{s_{p,\alpha}}_{p,p}(\mathbb{R}^3)\subset \dot{B}^{s_{p}}_{p,\infty}(\mathbb{R}^3)\cap L_{2}(\mathbb{R}^3).
$$
From this, along with the inclusion $\mathbb{\dot{B}}^{s_p}_{p,\infty}(\mathbb{R}^3)\hookrightarrow VMO^{-1}(\mathbb{R}^3)$, we infer
\begin{equation}\label{corollarybesovembeddings}
\mathbb{\dot{B}}^{s_p}_{p,\infty}(\mathbb{R}^3)\cap J(\mathbb{R}^3)\subset VMO^{-1}(\mathbb{R}^3)\cap \dot{B}^{s_{p,\alpha}}_{p,p}(\mathbb{R}^3)\cap J(\mathbb{R}^3).
\end{equation}
From (\ref{pcondition}) and (\ref{corollarybesovembeddings}), we infer that conclusions of Corollary \ref{cannoneweakstronguniqueness} is an immediate consequence of Theorem \ref{weakstronguniquenessBesov}.
\end{proof}
\begin{remark}\label{remarkcannoneweakstrong}
From the above proof of Corollary \ref{cannoneweakstronguniqueness}, we see that
\begin{equation}\label{besovinclusionforestnearinitialtime}
\dot{B}^{s_{p}}_{p,\infty}(\mathbb{R}^3)\cap J(\mathbb{R}^3)\subset\dot{B}^{s_{p}}_{p,\infty}(\mathbb{R}^3)\cap J(\mathbb{R}^3). 
\end{equation}
Hence, the conclusions of Lemma \ref{estnearinitialforLeraywithbesov} apply if $u_{0}\in \dot{B}^{s_{p}}_{p,\infty}(\mathbb{R}^3)\cap J(\mathbb{R}^3).$
Using this we claim the assumptions of Corollary \ref{cannoneweakstronguniqueness} can be weakened. Namely, there exists a small $\tilde{\epsilon}_0=\tilde{\epsilon}_0(p)$ such that  for all $u_{0} \in \dot{B}^{s_{p}}_{p,\infty}(\mathbb{R}^3)\cap J(\mathbb{R}^3)$ with
\begin{equation}\label{smallkatonorm}
\sup_{0<t<T} t^{\frac{-s_{p}}{2}}\|S(t)u_0\|_{L_{p}}<\tilde{\epsilon}_0, 
\end{equation}
we have the following implication. Specifically, all weak Leray-Hopf solutions on $Q_{\infty}$, with initial data $u_0$, coincide on $Q_{T}$. 

Indeed taking any fixed $u_0$ in this class and taking $\tilde{\epsilon_0}$ sufficiently small, we may argue in a verbatim fashion as in the proof of Theorem \ref{regularity} (setting $\delta=0$).  Consequently, there exists a weak Leray-Hopf solution $v$ on $Q_{T}$ such that
\begin{equation}\label{Katosmall}
\sup_{0<t<T}t^{\frac{-s_{p}}{2}}\|v(\cdot,t)\|_{L_{p}}<2\tilde{\epsilon}_0.  
\end{equation}

Using this and Lemma \ref{estnearinitialforLeraywithbesov}, we may argue in a similar way to the proof of Theorem \ref{weakstronguniquenessBesov} to obtain the desired conclusion.
\end{remark} 
\setcounter{equation}{0}
\section{Uniqueness Criterion for Weak Leray Hopf Solutions}
 Now let us state two known facts, which will be used in the proof of Proposition \ref{extendladyzhenskayaserrinprodi}.
 If $v$ is a weak Leray-Hopf solution on $Q_{\infty}$ with initial data $u_{0} \in J(\mathbb{R}^3)$, then this implies that $v$ satisfies the integral equation in $Q_{\infty}$:
  \begin{equation}\label{vintegeqnweakLerayhopf}
  v(x,t)= S(t)u_0+G(v\otimes v)(x,t). 
  \end{equation}
 The second fact is as follows. Consider $3<p<\infty$ and $2<q<\infty$ such that
 $
  {3}/{p}+{2}/{q}=1.
  $
  Then there exists a constant $C=C(p,q)$ such that for all $f, g\in L^{q,\infty}(0,T; L^{p,\infty}(\mathbb{R}^3))$
  \begin{equation}\label{bicontinuityLorentz}
  \|G(f\otimes g)\|_{L^{q,\infty}(0,T; L^{p,\infty}(\mathbb{R}^3))}\leq C\|f\|_{L^{q,\infty}(0,T; L^{p,\infty}(\mathbb{R}^3))}\|g\|_{L^{q,\infty}(0,T; L^{p,\infty}(\mathbb{R}^3))}.
  \end{equation}
 These statements and their corresponding proofs, can be found in \cite{LR1}\footnote{Specifically, Theorem 11.2 of \cite{LR1}.} and \cite{LRprioux}\footnote{Specifically, Lemma 6.1 of \cite{LRprioux}.}, for example.
\subsection{Proof of Proposition \ref{extendladyzhenskayaserrinprodi}}
\begin{proof}
\begin{itemize}
\item[] \textbf{Case 1: $u$ satisfies (\ref{extendladyzhenskayaserrinprodi2})-(\ref{smallness})  }
\end{itemize}
In this case, we see that the facts mentioned in the previous paragraph imply
\begin{equation}\label{lorentzinitialdatasmall}
\|S(t)u_0\|_{L^{q,\infty}(0,T; L^{p,\infty})}\leq \epsilon_{*}+C\epsilon_{*}^{2}.
\end{equation}
For $0<t_{0}<t_{1}$, we may apply O'Neil's convolution inequality (Proposition \ref{O'Neil}), to infer that
\begin{equation}\label{almostdecreasingverified}
\|S(t_1)u_0\|_{L^{p,\infty}}\leq 3p \|S(t_{0})u_0\|_{L^{p,\infty}}.
\end{equation}
This, in conjunction with (\ref{lorentzinitialdatasmall}), allows us to apply Lemma \ref{pointwiselorentzdecreasing} to obtain  that for $0<t<T$
\begin{equation}\label{lorentzpointwiseexplicit}
\|S(t)u_0\|_{L^{p,\infty}}\leq \frac{ 6p(\epsilon_{*}+C\epsilon_{*}^{2})}{t^{\frac{1}{q}}}.
\end{equation}
A further application of O'Neil's convolution inequality gives
\begin{equation}\label{lorentzbesovembedding}
\|S(t)u_0\|_{L_{2p}}\leq \frac{C'(p)\|S(t/2)u_0\|_{L^{p,\infty}}}{t^{\frac{3}{4p}}}
\end{equation}
This and (\ref{lorentzpointwiseexplicit}) implies that for $0<t<T$ we have
\begin{equation}\label{katoclassestimate}
\sup_{0<t<T} t^{\frac{-s_{2p}}{2}} \|S(t)u_0\|_{L_{2p}}\leq C''(p,q)(\epsilon_{*}+\epsilon_{*}^2).
\end{equation}
Recalling that $u_0 \in J(\mathbb{R}^3)$ ($\subset \mathcal{S}_{h}^{'}$), we see that by Young's inequality
$$ t^{\frac{-s_{2p}}{2}}\|S(t)u_0\|_{L_{2p}}\leq \frac{C'''(p)\|u_{0}\|_{L_{2}}}{t^{\frac{1}{4}}}.$$
From this and (\ref{katoclassestimate}), we deduce that
\begin{equation}\label{initialdataspaces}
u_{0} \in J(\mathbb{R}^3)\cap \dot{B}^{s_{2p}}_{2p,\infty}(\mathbb{R}^3).
\end{equation}
Using (\ref{katoclassestimate})-(\ref{initialdataspaces}) and Remark \ref{remarkcannoneweakstrong}, once reaches the desired conclusion provided $\epsilon_{*}$ is sufficiently small.
\begin{itemize}
\item[] \textbf{Case 2: $u$ satisfies (\ref{extendladyzhenskayaserrinprodi1})-(\ref{integrabilitycondition1})  }
\end{itemize}
The assumptions (\ref{extendladyzhenskayaserrinprodi1})-(\ref{integrabilitycondition1}) imply that $u\in L^{q,\infty}(0,T; L^{p,\infty}(\mathbb{R}^3))$ with  
\begin{equation}\label{shrinkingLorentz}
\lim_{S\rightarrow 0} \|u\|_{L^{q,\infty}(0,S; L^{p,\infty})}=0.
\end{equation}
From \cite{Sohr}, we have that  
\begin{equation}\label{sohrreg}
u \in L_{\infty}(\mathbb{R}^3 \times ]T',T[)\,\,\rm{for\,\,any}\,\, 0<T'<T.
\end{equation}
Now (\ref{shrinkingLorentz}), allows us to reduce to case 1 on some time interval. This observation, combined with (\ref{sohrreg}), enables us to deduce that $u$ is unique on $Q_{T}$ amongst all other weak Leray-Hopf solutions with the same initial value. 
 
\end{proof}

\setcounter{equation}{0}
\section{Appendix}
\subsection{Appendix I: Sketch of Proof of Theorem \ref{regularitycriticalbmo-1}}
\begin{proof}
\begin{itemize}
\item[]\textbf{ Step 1: the mollified  integral equation}
\end{itemize}
Let $\omega \in C_{0}^{\infty}(B(1))$ be a standard mollifier. Moreover, denote $$\omega_{\epsilon}(x):= \frac{1}{\epsilon^3} \omega\Big(\frac{x}{\epsilon}\Big).$$ 
Recall Young's inequality:
\begin{equation}\label{Youngsrecall}
\|f\star g\|_{L^{r}}\leq C_{p,q}\|f\|_{L^{p}}\|g\|_{L^{q}}
\end{equation}
where $1<p,q,r<\infty$, $0<s_1,s_2\leq\infty$, $\frac 1 p+\frac 1 q=\frac 1 r +1$.

Applying this and the pointwise estimate (\ref{kenrelKest})  gives the following:
\begin{equation}\label{Ktensormainest1recall}
\|K(\cdot, t-\tau)\star (f\otimes (\omega_{\epsilon}\star g))(\cdot,\tau)\|_{L_{2}}\leq  \|K(\cdot,t-\tau)\|_{L_1}\|f\otimes (\omega_{\epsilon}\star g)\|_{L_{{2}}}\leq$$$$\leq C (t-\tau)^{-\frac{1}{2}}\|f\otimes (\omega_{\epsilon}\star g)\|_{L_{2}}\leq
C (t-\tau)^{-\frac{1}{2}}\frac{\|f\|_{\mathcal{E}_T}\|g\|_{L_{\infty}(0,T; L_{2})}}{\tau^{\frac 1 2}}.
\end{equation}
 One can then show that 
\begin{equation}\label{Gtensorest1bmo-1}
\|G(f\otimes (\omega_{\epsilon}\star g))\|_{L_{\infty}(0,T; L_{2})}\leq C\|f\|_{\mathcal{E}_T}\|g\|_{L_{\infty}(0,T; L_{2})}.
\end{equation}
From \cite{LRprioux}, it is seen that $\mathcal{E}_{T}$ is preserved by the operation of mollification:
\begin{equation}\label{kochtatarupathspacemollify}
\| \omega_{\epsilon}\star g\|_{\mathcal{E}_{T}}\leq C'\|g\|_{\mathcal{E}_{T}}.
\end{equation}
Here, $C'$ is independent of $T$ and $\epsilon$. Using this and (\ref{bilinbmo-1}), we obtain:
\begin{equation}\label{Gtensorest2bmo-1}
\|G(f\otimes (\omega_{\epsilon}\star g))\|_{\mathcal{E}_{T}}\leq C\|f\|_{\mathcal{E}_T}\|g\|_{\mathcal{E}_T}.
\end{equation}

We briefly describe successive approximations.
 For $n=1,2,\cdots$ let $v^{(0)}= S(t)u_0$, $$v^{(n+1)}= v^{(0)}+G(v^{(n)} \otimes (\omega_{\epsilon}\star v^{(n)})).$$
 Moreover for $n=0,1,\ldots$ define:
 \begin{equation}\label{Mdefbmo-1}
 M^{(n)}:=\|v^{(n)}\|_{\mathcal{E}_T}
 \end{equation}
 and 
 \begin{equation}\label{Kdefbmo-1}
 K^{(n)}:=\|v^{(n)}\|_{L_{\infty}(0,T; L_{2})}.
 \end{equation}
 Then using (\ref{Gtensorest1bmo-1}) and (\ref{Gtensorest2bmo-1}),  we have the following iterative relations:
 \begin{equation}\label{Miterativebmo-1}
 M^{(n+1)}\leq M^{(0)}+C(M^{(n)})^2
 \end{equation}
 and 
 \begin{equation}\label{Kiterativebmo-1}
 K^{(n+1)}\leq K^{(0)}+CM^{(n)}K^{(n)}.
\end{equation}
 If 
 \begin{equation}\label{Msmallnessconditionbmo-1}
 4CM^{(0)}< 1,
 \end{equation}
 then one can show that for $n=1,2\cdots$ we have
 \begin{equation}\label{Mboundbmo-1}
 M^{(n)}<2M^{(0)}
 \end{equation}
and
 \begin{equation}\label{Kboundbmo-1}
 K^{(n)}<2K^{(0)}.
\end{equation}
Using (\ref{Miterativebmo-1})-(\ref{Kboundbmo-1}) and arguing as in \cite{Kato}, we see that there exists $v^{\epsilon}\in L_{\infty}(0,T;L_{2})\cap\mathcal{E}_{T}$ such that
 \begin{equation}\label{converg1bmo-1}
 \lim_{n\rightarrow\infty}\|v^{n}-v^{\epsilon}\|_{\mathcal{E}_{T}}=0,
 \end{equation}
 \begin{equation}\label{converg2bmo-1}
 \lim_{n\rightarrow\infty}\|v^{n}-v^{\epsilon}\|_{L_{\infty}(0,T;L_{2})}=0
 \end{equation}
 and $v^{\epsilon}$ solves the integral equation 
 \begin{equation}\label{integraleqnmollified}
 v^{\epsilon}(x,t)= S(t)u_0+G(v^{\epsilon}\otimes (\omega_{\epsilon}\star v^{\epsilon}))(x,t)
 \end{equation}
 in $Q_{T}$. Define $$\pi_{v^{\epsilon}\otimes (\omega_{\epsilon}\star v^{\epsilon})}:= \mathcal{R}_{i}\mathcal{R}_{j}(v^{\epsilon}_{i}(\omega_{\epsilon}\star v^{\epsilon})_{j}),$$ where $\mathcal{R}_{i}$ denotes the Riesz transform and repeated indices are summed. 
 One can readily show that on $Q_{T}$,   $(v^{\epsilon}, \pi_{v^{\epsilon}\otimes (\omega_{\epsilon}\star v^{\epsilon})})$ are solutions to the mollified Navier-Stokes system:
 \begin{equation}\label{vepsiloneqn}
 \partial_{t} v^{\epsilon}-\Delta v^{\epsilon} +\rm{div}\,(v^{\epsilon}\otimes (\omega_{\epsilon}\star v^{\epsilon}))=-\nabla \pi_{v^{\epsilon}\otimes (\omega_{\epsilon}\star v^{\epsilon})},
 \end{equation}
 \begin{equation}\label{vepsilondivfree}
 \rm{div\,v^{\epsilon}}\,=0,
 \end{equation}
 \begin{equation}\label{vepsilonintialcondition}
 v^{\epsilon}(\cdot,0)=u_0(\cdot).
 \end{equation}
 We can also infer that 
 $v^{\epsilon}\in W^{1,0}_{2}(Q_T)\cap C([0,T]; J(\mathbb{R}^3))$, along with the energy equality
 \begin{equation}\label{vepsilonenergyest}
\|v^{\epsilon}(\cdot,t)\|_{L_{2}}^2+\int\limits_0^t\int\limits_{\mathbb{R}^3}|\nabla v^{\epsilon}(x,s)|^2 dxds= \|u_0\|^2_{L_2}.
 \end{equation} 
 
 \begin{itemize}
 \item[]\textbf{Step 2: passing to the limit in $\epsilon\rightarrow 0$}
 \end{itemize}
 Since $u_0 \in VMO^{-1}(\mathbb{R}^3)\cap J(\mathbb{R}^3)$, it is known from \cite{LRprioux} (specifically, Theorem 3.5 there) that there exists a $v\in \mathcal{E}_{T}$ such that
 \begin{equation}\label{convergmild}
 \lim_{\epsilon\rightarrow 0}\|v^{\epsilon}- v\|_{\mathcal{E}_{T}}=0
 \end{equation}
 with $v$ satisfying the integral equation (\ref{vintegeqnbmo-1}) in $Q_{T}$. Using arguments from \cite{Le} and (\ref{vepsiloneqn})-(\ref{vepsilonenergyest}), we see that $v^{\epsilon}$ will converge to a weak Leray-Hopf solution on $Q_{T}$ with initial data $u_0$. Thus, $v\in \mathcal{E}_{T}$ is a weak Leray-Hopf solution.

 The remaining conclusions of Theorem \ref{regularitycriticalbmo-1} follow from similar reasoning as in the proof of the statements of Theorem \ref{regularity}, hence we omit details of them.
 \end{proof}

\textbf{Acknowledgement.} The author wishes to warmly thank Kuijie Li, whose remarks on the first version of this paper led to an improvement to the statement of Theorem \ref{weakstronguniquenessBesov}.

\end{document}